\documentclass[psamsfonts]{amsart}

\usepackage{amssymb,amsfonts}
\usepackage[all,arc]{xy}
\usepackage{enumerate}
\usepackage{mathrsfs}

\usepackage{mathtools}
\usepackage[utf8]{inputenc}
\usepackage[T1]{fontenc}
\usepackage{babel}
\usepackage{amsmath}
\usepackage{amsfonts}
\usepackage{amssymb}
\usepackage{amsthm}
\usepackage{verbatim}
\usepackage{todonotes}
\usepackage[caption=false]{subfig}
\usepackage{faktor}
\usepackage{emptypage}

\usepackage{enumerate} 

\usepackage[colorlinks,linkcolor=blue, citecolor=blue,urlcolor=blue]{hyperref} 
\usepackage{color}
\usepackage{blindtext}
\usepackage{soul}

\usepackage{graphicx} 
\usepackage{caption}
\usepackage{multirow}

\usepackage{amssymb,tikz}
\usetikzlibrary{arrows.meta, positioning}

\usepackage{mathrsfs}

\usepackage{amsmath,amssymb,tikz-cd}
\usepackage{adjustbox}

\def\hbgi{H^{*}(BG_i; \mathbb{F}_p)}

\newcommand\exbegin{\begin{center}\begin{minipage}{10cm}\begin{example}\setlength{\parskip}{1.5ex}\small\rm}
\newcommand\exend{\end{example}\end{minipage}\end{center}}

\def\Hom{\operatorname{Hom}}
\def\leng{\operatorname{length}}
\newcommand{\Mod}{\operatorname{Mod}}
\newcommand{\End}{\operatorname{End}}
\newcommand{\Rep}{\operatorname{Rep}}

\newtheorem{Teo}{Theorem}[section]
\newtheorem{Prop}[Teo]{Proposition}
\newtheorem*{Coro.l}{Corollary}
\newtheorem{Coro}{Corollary}[Teo] 
\newtheorem{Lema}[Teo]{Lemma}

\theoremstyle{definition}
\newtheorem{Def}{Definition}[section]
\newtheorem*{Not}{Notation}
\newtheorem*{Obs}{Remark}
\newtheorem{Ex}{Example}[section]

\theoremstyle{remark}

\newcommand{\hocolim}{\operatorname{hocolim}}
\newcommand{\colim}{\operatorname{colim}}
\newcommand{\CP}{\mathcal{P}}
\newcommand{\CQ}{\mathcal{Q}}
\newcommand{\Cat}{\operatorname{Cat}}
\newcommand{\Tr}{\operatorname{Tr}}

\makeatletter
\let\c@equation\c@Teo
\makeatother
\numberwithin{equation}{section}

\title{Cohomological properties of simple complexes of groups}

\author{Roger Bergadà Batlles}

\date{}

\begin{document}

\begin{abstract}

A simple complex of groups is a functor from a poset to the category of groups satisfying certain properties. They were introduced by Bridson and Haefliger in their book \cite{BH} as an abstract way to encode isotropy subgroups of actions on posets. A simplex complex of groups which embeds in a group $K$ has an associated $K$-poset. We will explain how to obtain information about the classifying space of $K$ through the $K$-equivariant homotopy type of the construction of Bridson-Haefliger, and we will also obtain finiteness properties of cohomology.

\end{abstract}

\maketitle

\tableofcontents

\section*{Introduction}

Simple complex of groups give rise to the concept of complexes of groups, which is due to Haefliger and provides a natural generalization of the concept of a graph of groups, introduced by Bass and Serre.

We show that their study can be applied to interesting cases such as Coxeter groups, which form a fundamental class of groups appearing in many areas of combinatorics, geometry, topology, and group theory, as in the theory of Lie algebras or the theory of buildings.

Using the construction of the associated $K$-poset to a simple complex of groups, we obtain descriptions of the corresponding Borel constructions as homotopy co\-limits, allowing us to apply homotopical and spectral sequence techniques to their study. Our motivation, then, is to understand the algebraic topology properties of the group $K$ associated to a simple complex of groups based on these same properties of the Borel construction of the $K$-poset associated to the simple complex of groups.

The main objective of the thesis is to establish finiteness properties for the cohomology of these constructions. More specifically, we prove finite generation results for mod $p$ cohomology and extend these methods to cohomology with other coefficients. The strategy combines several ingredients from algebraic topology and representation theory, including the Bousfield--Kan spectral sequence, Quillen's $F$-isomorphism theorem, and the theory of unstable algebras over the Steenrod algebra. A key technical ingredient is the study of colimits of Noetherian $\mathrm{End}(V)$-sets, which reflect finiteness conditions for unstable algebras, and we do this using some theorems from Henn, Lannes and Schwartz.

\textbf{Outline of this paper:} This paper is organized as follows. Section \ref{Preliminaries}, collects the background material required for the sequel. We review some results of ho\-mo\-to\-py theory and cohomology of groups through Section \ref{prelimhomotopy}. In Section \ref{modulesoversteen}, we recall the necessary facts about the Steenrod algebra and unstable modules and algebras, where the theory of $\mathrm{End}(V)$-sets ($V$ an elementary abelian $p$-group) by Henn-Lannes-Schwartz is introduced, and we prove some results concerning Noetherian $\mathrm{End}(V)$-sets, that will be used repeatedly later. We also introduce the notion and some properties of the simple complexes of groups in Section \ref{simplecomplex}. Then, in Section \ref{Homotopy}, we develop the homotopy theory of simple complexes of groups and establish homotopies for the realization of the associated $K$-poset and for its Borel construction using homotopy colimits. We also describe, in Section \ref{sec:join}, the join construction for simple complexes of groups. We then prove, in Section \ref{C:ModpCohomology}, the main theorem of this paper, which establishes finite generation results for mod $p$ cohomology of the Borel construction via $F$-isomorphisms, using spectral sequence techniques, $\textrm{End}(V)$-sets theory and Steenrod algebra. Firstly, we prove it for simple complexes of groups using a finite poset as the indexing category. We then consider the more general case of an infinite indexing poset, and finally give a brief study of cohomology with $p$-adic coefficients.

\textbf{Acknowledgments:} The author thanks Natàlia Castellana for her guidance and support throughout this work. The author was supported by MICINN grants PID2024-158573NB-I00 and PID2020-116481GB-I00. Part of this article is based on the author's Ph.D. thesis, supported by the FI-SDUR grant 2022-FISDU-00399.

\section{Preliminaries}\label{Preliminaries}

The following sections recall some results of homotopy theory and cohomology of groups, introduce concepts about Steenrod algebra and operations with Noetherian \textrm{End}(V)\text{-sets}. Finally, we introduce relevant definitions about simple complexes of groups.

\subsection{Homotopy theory and cohomology}\label{prelimhomotopy}
The Borel construction plays a cen\-tral role in this paper. 

\begin{Def}\label{clasicborelcons}
Let $X$ be a $G$-space. The \emph{Borel construction} of $X$ is 
\[
X_{hG}:=EG\times_G X,
\]
the orbit space of the diagonal $G$-action.
\end{Def}

Moreover, the projection onto the first factor of the Borel construction with the universal principal $G$-bundle $EG\to BG$ induces the following fibration
\[
X
\longrightarrow
X_{hG}
\longrightarrow
BG,
\]
known as the \emph{Borel fibration}.

Alternatively, every group $G$ can be regarded as a category with a single object, which we denote by $\ast$. The morphisms of the category, $\ast \to \ast$, are precisely the elements of $G$, with composition given by the group operation. Then we can introduce the Borel construction as an homotopy colimit.
\begin{Obs}
Let $X$ be a $G$-space. Consider the functor $F\colon G\longrightarrow \mathbf{Top},$ where $G$ is regarded as a category, $F(\ast)=X$, and each element $g\in G$ is sent to the homeomorphism
\[
F(g):\begin{array}[t]{rcl}
X & \longrightarrow & X\\
x & \longmapsto & gx.
\end{array}
\]
The \emph{Borel construction} of $X$ can also be defined as the homotopy colimit of the functor
$F$ (see \cite[Section 2]{WZ}):
\[
X_{hG}=\operatorname{hocolim}_G F.
\]
\end{Obs}

Let $I$ be an indexing category, and let $\Cat$ be the category of small categories. Let \(F \colon I \to \Cat\) be a functor, then, another important construction will be the \emph{Grothendieck construction} $\Tr(F)$ (see \cite{Thomason}). In particular, we will use the following theorem.

\begin{Teo}[Homotopy colimit theorem {(\cite[Theorem 1.2]{Thomason})}]\label{simeq:hocolimTr}
Let $F\colon I\to \Cat$ be a functor. There is a natural homotopy equivalence
\[
\eta :
\operatorname{hocolim}_I NF
\longrightarrow
N(Tr(F))
\]
of the homotopy colimit of \(NF\) and the nerve of the Grothendieck
construction.
\end{Teo}

Also, in the cohomological study, we should recall the following theorem that establishes the finite generation of the cohomology ring of a finite group.
\begin{Teo}[Evens--Venkov {(\cite[Corollary 7.4.6]{Evens_1991})}]\label{fgRalg}
Let $G$ be a finite group and let $R$ be a Noetherian ring on which $G$ acts trivially. Then the cohomology ring $H^*(G;R)$ is a finitely generated $R$-algebra.
\end{Teo}

\subsection{Modules and algebras over the Steenrod algebra $\mathcal A_p$}\label{modulesoversteen}

In this subsection, we will recall the construction and some important properties of the Steenrod algebra, which will be very useful later in this paper. The material in this subsection is standard and is based primarily on \cite{Schwartz_1994}, see \cite[Section 1]{Schwartz_1994} for an intro\-duc\-tion to the Steenrod algebra $\mathcal{A}_p$. We will mainly use it to provide extra structure on the cohomology of spaces, as is specified by the following theorem, due to Steenrod and Adem. The mod $p$ cohomology $H^*(X; \mathbb{F}_p)$ of a space $X$ will be denoted by $H^*X$.

\begin{Teo}\cite[Theorem 1.1.1]{Schwartz_1994}
For any space $X$, $H^*X$ is in a natural way a graded unstable algebra over $\mathcal{A}_p$.
\end{Teo}

Let $\mathcal U$ (resp. $\mathcal K$) be the category of unstable modules (resp. algebras) over the Steenrod algebra. We recall some standard notions that will be used throughout this paper. 

There are several notions to measure the 'finiteness' of an algebra. The more geometric one is the Krull dimension: the supremum of the lengths of all chains of prime ideals. This notion can be read from the spectrum of homogeneous prime ideals. Another measure is the transcendence degree.

\begin{Def}\label{def:transcendence}
Let $K$ be a graded $\mathbb{F}_p$-algebra. The transcendence degree of $K$, $d(K)$, is the supremum of the cardinalities of finite subsets of algebraically independent homogeneous elements of $K$. The transcendence degree of an unstable algebra $K$ is the transcendence degree of its underlying graded algebra.
\end{Def}

\begin{Obs}
    If $K$ is Noetherian, then both notions agree, that is $d(K)$ is the Krull dimension. But in general, the Krull dimensions is less than or equal to the transcendence degree.
\end{Obs}

A fundamental result in the study of the mod $p$ cohomology of groups is Quillen's $F$-isomorphism theorem (see \cite[Theorem 6.2]{quillen}). This theorem shows that the cohomology of a finite group can be detected, up to nilpotent information, by restricting to its elementary abelian $p$-subgroups. 

\begin{Teo}[Quillen's $F$-isomorphism theorem]\label{quillen:fiso}
Let $G$ be a finite group and let $\mathcal A_p(G)$ denote the category of elementary abelian $p$-subgroups of $G$, and morphisms induced by conjugation. Then the homomorphism
\[
H^*(BG;\mathbb F_p)\longrightarrow
\varprojlim_{E\in \mathcal A_p(G)} H^*(BE;\mathbb F_p)
\]
is an $F$-isomorphism.
\end{Teo}
Since $H^*(BG;\mathbb F_p)$ is a Noetherian $\mathbb{F}_p$-algebra (Theorem \ref{fgRalg}), Quillen proved the following.

\begin{Coro}
   Let $G$ be a finite group. The Krull dimension of $H^*(BG;\mathbb F_p)$ is the transcendence degree and $d(H^*(BG;\mathbb F_p))$ is the maximum of the ranks of  elementary abelian $p$-subgroups of $G$.
\end{Coro}

As we can find in \cite{LANNES1989153}, we can define a nilpotent unstable $\mathcal A_p$-module $M$ using Steenrod operations in the absence of a product. The following definition and property will be useful to encode this notion of nilpotence.

\begin{Def}
    Let $M$ be an unstable module over $\mathcal{A}_p$. Let $m\in M$. If $p=2$, define $Sq_0(m)=Sq^{|m|}(m)$. If $p>2$ define $\mathcal{P}_0(m)=\mathcal{P}^{|m|/2}(m)$ if $|m|$ is a multiple of $p$.
\end{Def}

\begin{Prop}\label{sq0}\cite[Section 1.7]{Schwartz_1994}
    Let $M$ be an unstable module over $\mathcal{A}_p$. Let $m\in M$. Then, if $p=2$, $Sq^iSq_0(m)=Sq_0Sq^{|m|/2}(m)$ when $|m|$ is even and $0$ otherwise. If $p>2$, then $\mathcal{P}^i\mathcal P_0(m)=\mathcal{P}_0\mathcal P^{|m|/p}(m)$ when $i$ is a multiple of $p$. 
\end{Prop}

We also recall the characterization of Noetherian up to $F$-isomorphism by Henn-Lannes-Schwartz \cite{HLSmodulonilpotents}, where as before, $H^*(V)$ denotes the mod $p$ cohomology $H^*(V; \mathbb{F}_p)$ of the elementary abelian $p$-group $V$. 

\begin{Prop}\label{prop:FisoEndViso}
    A morphism $f\colon K\rightarrow L$ of unstable $\mathcal{A}_p$-algebras is an $F$-iso\-mor\-phism if and only if the induced function $$\mathrm{Hom}_\mathcal K(L,H^*(V))\cong \mathrm{Hom}_\mathcal K(K,H^*(V))$$ is a bijection for all elementary abelian $p$-group $V$. 
\end{Prop} 

\begin{Def}
If $K$ is an unstable $\mathcal{A}_p$-algebra over the Steenrod algebra, then for any elementary abelian $p$-subgroup $V$ of rank $d$, define the $\textrm{End}(V)$-set $$s_d(K):=\mathrm{Hom}_\mathcal K(K,H^*(V)).$$ 
\end{Def}

\begin{Def}
    Let $S$ be an $\textrm{End}(V)$-set. Define  
    $$b_d(S)\colon =\Hom_{\textrm{End}(V)}(S,H^*(V)),$$ which is a graded unstable $\mathbb{F}_p$-algebra over $\mathcal{A}_p$, where the algebraic structure and the action of $\mathcal{A}_p$ come from the target.
\end{Def}

\begin{Def}\label{kerS}
    If $S$ is an $\textrm{End}(V)$-set and $s\in S$, the kernel $\ker (s)\leq V$ is the subgroup uniquely characterized by the following properties:
    \begin{enumerate}
        \item For any $t\in S$ and $\alpha \in \textrm{End}(V)$ such that $s=t\alpha$, $\ker \alpha \subseteq \ker s$.
        \item There are elements $t_0\in S$ and $\alpha_0 \in \textrm{End}(V)$ with $s=t_0\alpha_0$ and $\ker \alpha_0={\ker}(s)$.
        
    \end{enumerate}
\end{Def}

\begin{Def}
    An $\End(V)$-set $S$ is Noetherian if $S$ is finite and $\alpha^{-1}(\ker (s))=\ker(s\alpha)$ for any $s\in S$ and $\alpha \in \textrm{End}(V)$, where $\ker(s)$ is the maximal kernel among all endomorphisms in the orbit of $s$ (see Definition \ref{kerS}).
\end{Def}

The connection between unstable algebras over the Steenrod algebras and finite\-ness conditions on $\textrm{End}(V)$-sets is described in \cite{HLSmodulonilpotents}, and summarized in the following statement.

\begin{Teo}\cite[Theorem 7.1]{HLSmodulonilpotents}\label{b_dTheo}
     If $K$ is a Noetherian unstable $\mathcal{A}_p$-algebra, then $s_d(K)$ is a Noetherian $\End(V)$-set. And if $S$ is a Noetherian $\End(V)$-set, then $b_d(S)$ is a Noetherian unstable $\mathcal{A}_p$-algebra of transcendence degree at most $d$.
\end{Teo}

We will use the following property.

\begin{Teo}\cite{HLSmodulonilpotents}\label{bdsd:fiso}
    Given an unstable $\mathcal{A}_p$-algebra $K$ of transcendence degree less than $d$, then $K\to b_d(s_d(K))$ is an $F$-isomorphism.
\end{Teo}

We study colimits and find conditions under which the colimit of Noetherian $\textrm{End}(V)$-sets is again Noetherian. This will allow us to prove that the colimit construction preserves the algebraic finiteness properties of the mod $p$ cohomology of simple complexes of groups.

\begin{Def}Let $X$ and $Y$ be two $\textrm{End}(V)$-sets.
    \begin{enumerate}
    \item An \emph{$\textrm{End}(V)$-map} $f\colon X\to Y$ is a map such that for every $x\in X$ and $\alpha\in \textrm{End}(V)$, we have $f(x\alpha)=f(x)\alpha$.
    \item We say that an $\textrm{End}(V)$-map $f\colon X\to Y$ \emph{preserves kernels} if $\ker(x)=\ker(f(x))$ for any $x\in X$ \cite[4.5]{BK:Kac-Moody_groups}.
    \end{enumerate}
\end{Def}

We state the main result:

\begin{Prop}\label{p:colimNoetherian}
    Let  $\mathcal Q$ be a finite poset. Let $X\colon\mathcal{Q}\to \textrm{End}(V)\text{-sets}$ be a functor such that $X(q)$ is a Noetherian $\End(V)$-set for every $q\in \mathcal{Q}$, and $X(i)\colon X(q)\rightarrow X(q')$ is a morphism of $\End(V)$-sets which preserves kernels for every $i\colon q\to q'$ in $\mathcal{Q}$. Then $\colim_\mathcal{Q} X$ is a Noetherian $\End(V)$-set. Moreover, $X(q)\to\colim_\mathcal{Q} X$ is also a morphism of $\End(V)$-sets which preserves kernels. 
\end{Prop}

Before proving Proposition \ref{p:colimNoetherian}, we establish a technical lemma describing the kernel of an element in the colimit. The result shows that the kernel of a class in $\operatorname*{colim}_{Q} X$ can be recovered from the maximal kernels of its representatives in the different $\mathrm{End}(V)$-sets appearing in the diagram.

\begin{Lema}\label{kercolim}
     Let $X\colon\mathcal{Q}\to \textrm{End}(V)\text{-sets}$ be a functor. Consider $\bar{x}\in\colim_\mathcal{Q} X$ with $x\in X(q)$. Then $$\ker (\bar{x})= \textrm{Max}\{\ker(y) \mid \bar y = \bar x, y\in X(q_y), q_y\in \mathcal Q\}.$$
\end{Lema}

\begin{proof}
    We check it by double inclusion.
\end{proof}

\begin{proof}[Proof of Proposition \ref{p:colimNoetherian}]

We need to prove that for every $\bar{x}\in\colim_\mathcal{Q} X$, and $\alpha\in \textrm{End}(V)$, we have that 
\begin{equation}\label{noeth}
    \alpha^{-1}(\ker( \bar{x}))=\ker(\bar{x}\alpha).
\end{equation}
By Lemma \ref{kercolim}, $$\ker (\bar{x})= \textrm{Max}\{\ker(y) \mid \bar y = \bar x, y\in X(q_y), q_y\in \mathcal Q\}.$$

Furthermore, as $\mathcal{Q}$ is a finite poset, for every $x, y$ such that $\bar x=\bar y$, it exists $t\in X(q_t)$ such that $x=i_1(t)$ and $y=i_2(t)$. Now, as for every $q,q'\in \mathcal{Q}$, the morphism $i\colon X(q)\to X(q')$ preserves kernels, we have $\ker (y)=\ker(x).$ So for every $\bar{x}\in\colim_\mathcal{Q} X$, $\ker (\bar{x})=\ker(x)$. 

Now, we assume that $X(q)$ is a Noetherian $\End(V)$-set, so for any element $z\in X(q)$ and $\alpha \in \textrm{End}(V)$, we have $\alpha^{-1}(\ker (z))=\ker(z\alpha)$. Thus to prove \ref{noeth}, we just need to show that $$\alpha^{-1}(\ker( \bar{x}))=\alpha^{-1}(\ker(x))=\ker(x\alpha)=\ker(\overline{x\alpha})=\ker(\bar{x}\alpha).$$

The colimit can be described as $$\colim_\mathcal{Q} X := \bigsqcup_{q\in \mathcal{Q}} X(q)/\sim$$ where $(x,q)\sim (i^*(x),q')$ with $q\leq^i q'$, $x\in X(q)$ and $i^*(x)\in X(q')$.

Now, we know this is a Noetherian $\End(V)$-set,
and for every $q\in\mathcal{Q}$, we will see that the inclusion $\bar i\colon X(q)\to\colim_\mathcal{Q} X$ preserves kernels (i. e. $\ker(x)=\ker(\bar i(x))=\ker(\bar x)$ for any $x\in X(q)$). But by the Lemma \ref{kercolim},  $$\ker (\bar{x})= Max\{\ker(y) \mid \bar y = \bar x, y\in X(q_y), q_y\in \mathcal Q\},$$ so it's trivial that $\ker(x)\subseteq \ker (\bar{x})$.

But as we proved, since $\mathcal{Q}$ is a finite poset, for every $x\in X(q)$, we have $\ker(x)=\ker(\bar{x})$.
\end{proof}

\subsection{Simple complexes of groups}\label{simplecomplex}
In this section, we review the definition of simple complexes of groups and develop some basic constructions that will be used later. We review the definitions from \cite[II.12]{BH}.

Given a poset $\CP$, we can construct a small category whose objects are the elements in $\CP$, and there is a single morphism $p\to q$ if and only if $p\le q$. We abuse notation and we denote by $\CP$ also the poset seen as a category.

\begin{Def}\label{lengposet}
The \emph{length} of a poset is the size of its longest chain minus 1. Thus, $$\leng(\CP)=\sup\{|C|-1\mid C\subseteq\CP \text{ is a chain}\}.$$
\end{Def}

\begin{Def}\label{d:simplecomplexgroups}
Let $\mathcal Q$ be a connected poset, a \emph{simple complex of groups} is a functor $G\colon \mathcal Q\rightarrow \textrm{InjGroups}$, where $\textrm{InjGroups}$ is the category of groups and injective morphisms.  Two systems are isomorphic if there is a natural transformation that is an object-wise isomorphism.  
\end{Def}

\begin{Def}\label{associated}
Let $G\colon \mathcal Q\rightarrow \textrm{InjGroups}$ be a simple complex of groups, the associated group is $\hat{G}$ defined as the direct limit (colimit) of the functor, $\hat{G}=\colim_{\mathcal Q}G$. 
\end{Def}

\begin{Obs}
    The morphisms $\iota_q \colon G(q)\rightarrow \hat{G}$ are not necessarily injective.
\end{Obs}

The following example shows that, even when the canonical morphisms from the local groups to the associated group are injective, there may still exist non-injective homomorphisms from the associated group whose restrictions to all local groups remain injective.

\begin{Ex}
    Let $\mathcal{Q}=\{b>a<c\}$, and $F:\mathcal Q\rightarrow \textrm{InjGroups}$ given by $\{\mathbb Z \leftarrow 0 \rightarrow \mathbb Z\}$. Then $\hat{G}=\mathbb Z*\mathbb Z$. The abelianization morphism $\phi\colon \mathbb Z*\mathbb Z \rightarrow \mathbb Z\oplus\mathbb Z$ is not injective but it satisfies that the restriction to each local group is injective.
\end{Ex}

Let $G\colon \mathcal Q\rightarrow \textrm{InjGroups}$ be a simple complex of groups, and $\varphi\colon \hat{G}\rightarrow K$ a group homomorphism such that the restriction to $G(q)$ is injective for each $q\in \mathcal Q$, that is, $\varphi|_{G(q)}\coloneqq\varphi\circ \iota_{q} \colon G(q)\rightarrow K$ is a monomorphism. In this case, we abuse notation and consider $G(q)$ identified with $\varphi(G(q))\subset K$.

Coxeter groups form a fundamental class of groups appearing in many areas of geometry, topology, and group theory, and their rich combinatorial structure makes them particularly well suited for constructions involving posets and group actions.

\begin{Ex}\label{ex:coxeter}
    Let $W$ be a Coxeter group. Let $S=\{r_1,\dots , r_n\}$ be the set that generates $W$ (with some relations). Given $T\subset S$, let $W_T$ be the Coxeter subgroup generated by $T$ with the corresponding relations. Then the homomorphism $W_T \to W$ induced by the inclusion $T \xhookrightarrow{} S$ is injective. Thus, we can consider $\mathcal{Q}$ to be the poset of proper subsets of $S$ ordered by inclusion. Then, we define a corresponding simple complex of groups $$W\colon \mathcal{Q}\rightarrow \textrm{InjGroups},$$ where for every $T\subset S$, $W(T)=W_T$. This is called the Coxeter complex of $W$.
\end{Ex}

\section{Homotopy theory of simple complexes of groups}\label{Homotopy}

In this section, we study simple complexes groups from a homotopical theoretic point of view by describing the homotopy orbit space as a new poset and as a homotopy colimit.

\subsection{The equivariant homotopy type of the associated poset}\label{SectionHomotopy}

\begin{Def}\label{def:associatedGspace}
 Let $\varphi\colon \hat{G}\rightarrow K$ be a group homomorphism such that the restriction to $G(q)$ is injective for each $q\in \mathcal Q$. We define a $K$-poset $\mathcal P=\bigsqcup_{\mathcal Q} K/G(q)$ with partial ordering $(kG(q),q)<(k'G(r),r)$ if and only if $q<r$ and $k^{-1}k'\in G(r)$ and an action of $K$ via $k\cdot (gG(q),q)=(kgG(q),q)$. 
\end{Def}

Using this definition, we can state the main result of this section, which studies the homotopy of the geometrical realization of the poset $\CP$ we just defined and its Borel construction. We find homotopy equivalences between them and homotopy colimits of functors. These equivalences will help us, in the following section, to study the cohomology of $BK$ through the cohomology of the Borel construction of $\CP$ and $G(q)\subset K$.

\begin{Teo}\label{mainHomotopy}
Let $G\colon \mathcal Q\rightarrow \textrm{InjGroups}$ be a simple complex of groups, and $\mathcal P$ the associated poset to $\varphi\colon \hat{G}\rightarrow K$. There is a functor $F\colon \mathcal Q\rightarrow \Cat$ such that $$|\mathcal P|\simeq \hocolim_{\mathcal Q} |F|$$ with $|F(q)|\simeq K/G(q)$. Moreover, there is a fibration 
    $$|\mathcal P|\rightarrow |\mathcal P|_{hK}\rightarrow BK$$
    where $|\mathcal P|_{hK}\simeq \hocolim_{\mathcal{Q}^{op}} EK\times_K F(q)$, with $EK\times_K F(q)\simeq BG(q)$.
\end{Teo}

To prove this main theorem, we first establish some technical lemmas.

\begin{Lema}
 Let $\mathcal P$ be as in Definition \ref{def:associatedGspace}. There is a bijection between the connected components of $\mathcal P$ and the cosets $K/Im(\varphi)$.
\end{Lema}

\begin{proof}
Let $k_0\in K$.  Note that for every $q,r\in \mathcal{Q}$, we have $$(k_0G(q),q)<(k_0G(r),r)\iff q<r,$$ because the condition $k_0^{-1}k_0=e\in G(r)$ is always satisfied. Then $k_0$ determines a subposet $\{(k_0G(q),q) \mid q\in\mathcal{Q}\}$ isomorphic to $\mathcal{Q}$ which is connected, and then all these elements lie in the same connected component of $\mathcal{P}$. 
    
    The proof will be complete if we show that $(k'G(q),q)$ is in this same connected component if and only if $k'=k_0k$ with $k\in Im(\varphi)$. We introduce the following notation $(hG(q),q)\sim (h'G(q'),q')$ if and only if they lie in the same connected component.

     The subgroup $Im(\varphi) < K$ is generated by $\varphi(G(q))\cong G(q)$, for all $q\in\mathcal{Q}$. If $k\in Im(\varphi)$, we write $k=k_1\cdots k_n$ where $k_i\in G(q_i)$. Assume first that $k'=k_0k$ with $k\in Im(\varphi)$. Thus, 
    \begin{align*}
        (k'G(q),q)&=(k_0kG(q),q)=(k_0k_1\cdots k_n G(q),q)\sim(k_0k_1\cdots k_n G(q_n),q_n)=\\
        &=(k_0k_1\cdots k_{n-1} G(q_n),q_n)\sim \dots \sim (k_0k_1 G(q_1),q_1)=\\
        &=(k_0 G(q_1),q_1)\sim (k_0G(q),q)
    \end{align*}
    Now assume that $(k'G(q),q)\sim (k_0G(q),q)$, then there exists $l_1,\dots, l_n \in K$ and $q_1,\dots, q_n \in \mathcal{Q}$ such that

\begin{tikzcd}[column sep=tiny]
                             & {(l_1 G(q_1),q_1)} &                                                     & \dots \arrow[rd, draw=none, ">"{description, sloped} ] &                                    & {(k_0 G(q),q)} \\
{(k'G(q),q)} \arrow[ru, draw=none, "<"{description, sloped}] &                    & {(l_2 G(q_2),q_2)} \arrow[lu, draw=none, ">"{description, sloped}'] \arrow[ru, draw=none, "<"{description, sloped}] & \dots                 & {(l_n G(q_n),q_n)} \arrow[ru, draw=none, "<"{description, sloped}] &               
\end{tikzcd}
    
    and then satisfying $$l_1^{-1}k'\in G(q_1),$$ $$k_0^{-1}l_n\in G(q),$$ and $$l_i^{-1}l_{i-1}\in G(q_i)$$ for every $i\in \{2,\dots,n\}.$

Then, we have 
$$k_0^{-1}k'=(k_0^{-1}l_n) (l_n^{-1}l_{n-1})\cdots (l_1^{-1}k')\in Im(\varphi),$$
since it is a product of elements of $G(q_i)$ for some $q_i\in \mathcal{Q}$.
\end{proof}

\begin{Lema}\label{lem:coset as poset}
    Let $H\subset K$ be a subgroup of $K$. Consider the poset $\mathcal{P}^K_H$ whose objects are $k\in K$, and there is a unique morphism from $k$ to $kh$ for any $h\in H$. Then there is an equivalence of categories between $\mathcal{P}^K_H$ and $K/H$ considered as a category with no non-identity morphisms. In particular, its geometric realization is $|\mathcal{P}^K_H|\simeq K/H$.
\end{Lema}

\begin{proof}
    Note that we can describe the poset $\mathcal{P}^K_H$ as the disjoint union of $(\mathcal{P}^{K}_H)_{kH}$ for all the cosets $kH$ with $[k]\in K/H$.
    
Furthermore, as for any $h\in H$ there is a unique morphism from $k$ to $kh$, we get that for any $k\in K$, $|(\mathcal{P}^{K}_H)_{kH}|\cong |(\mathcal{P}^{K}_H)_{eH}|\simeq *$. Then we have that
    \begin{equation*}
        |\mathcal{P}^{K}_H|=\bigsqcup_{[k]\in K/H}|(\mathcal{P}^{K}_H)_{kH}|\simeq \bigsqcup_{[k]\in K/H} * = K/H.
    \end{equation*}
\end{proof}

Using this lemma, we can prove the first part of Theorem \ref{mainHomotopy}, recorded as the following proposition.

\begin{Prop}\label{prop:Pishocolim}
    Let $G\colon \mathcal Q\rightarrow \textrm{InjGroups}$ be a simple complex of groups, and $\mathcal P$ the associated poset to $\varphi\colon \hat{G}\rightarrow K$. There is a functor $F\colon \mathcal Q\rightarrow \Cat$ such that $$|\mathcal P|\simeq \hocolim_{\mathcal Q} |F|$$ with $|F(q)|\simeq K/G(q)$.
\end{Prop} 

\begin{proof}
    Consider the functor $F\colon \mathcal Q \rightarrow \Cat$ given by $F(q)=\mathcal{P}^K_{G(q)}$ as defined in Lemma \ref{lem:coset as poset}. If $\alpha \colon q\rightarrow q'$ in $\mathcal Q$ then $F(\alpha)\colon F(q)\rightarrow F(q')$ is the functor defined as the identity on objects, and it is well defined in morphisms since $G(q)\subset G(q')$. That's because in $F(q)$, we have a morphism $k_1\rightarrow k_2$ if and only if $k_1^{-1}k_2\in G(q)$. Then, since $G(q)\subset G(q')$, the functor $F(\alpha)$ is well defined in morphisms because we also have $k_1^{-1}k_2\in G(q')$, so we have a morphism $k_1\rightarrow k_2$ in $F(q')$.
    
    By Lemma \ref{lem:coset as poset},
    the geometric realization is homotopically discrete and $|F(q)|\simeq K/G(q)$.
    We then have $$\hocolim_{q\in \mathcal Q} K/G(q)\simeq \hocolim_{q\in \mathcal Q} |F(q)|.$$

    We can use Thomason's theorem \ref{simeq:hocolimTr} to describe this space as the nerve of poset. Thus, $\hocolim_{q\in \mathcal Q} |F(q)|\simeq |Tr(F)|$. We finish the proof by showing that there is an equivalence of categories $\mathcal P\rightarrow Tr(F)$.
    
    The objects of $Tr(F)$ are the pairs $(q,x)$ such that $q \in \mathcal Q$ and $x\in F(q)$, and note that $|F(q)|\simeq K/G(q)$. But as we have defined $\mathcal P$, it has by objects the pairs $(q,x)$ such that $q \in \mathcal Q$ and $x\in K/G(q)$, so $\mathcal P$ and $Tr(F)$ are equivalent in terms of objects. 

    Now, the morphisms of $Tr(F)$ are the pairs $(\alpha,\beta)\colon (q_0,x_0)\rightarrow (q_1,x_1)$ given by $\alpha \colon q_0 \rightarrow q_1$ (a morphism of $\mathcal Q$) and $\beta\colon F(\alpha)(x_0)\rightarrow x_1$ (a morphism of $F(q_1)$).
    
    As we have defined $\mathcal P$, it has by morphisms the pairs $$(\alpha,\beta)\colon (q_0,[k_0])\rightarrow (q_1,[k_1])$$ (with $[k_0]\in K/G(q_0)$ and $[k_1]\in K/G(q_1)$) given by $\alpha \colon q_0 \rightarrow q_1$ (a morphism of $\mathcal Q$) and $\beta\colon k_0 \rightarrow k_1$ (a morphism of $\mathcal{P}^K_{G(q_1)}$), which exist as long as $k_0^{-1}k_1\in G(q_1)$.
    
Then as we have defined the functor $F$ as the one given by $F(q)=\mathcal{P}^K_{G(q)}$, it's clear than the morphisms of $\mathcal P$ and $Tr(F)$ are also equivalent.

In order to check it formally, one can define two functors $T\colon \mathcal P \to Tr(F)$ and $S\colon Tr(F)\to \mathcal P$ such that $S\circ T \simeq id_\mathcal P$ and $T\circ S \simeq id_{Tr(F)}$.

\end{proof}

\begin{proof}[Proof of Theorem \ref{mainHomotopy}]
By Proposition \ref{prop:Pishocolim}, we already proved the first part of theorem. To finish the proof, since $\mathcal{P}$ is a $K$-poset, we may form its Borel construction as in Definition \ref{clasicborelcons}, $|\mathcal P|_{hK}=EK\times_K |\mathcal P|$, and it induces the Borel fibration

\[
|\mathcal P|\longrightarrow |\mathcal P|_{hK}\longrightarrow BK.
\]

By Proposition \ref{prop:Pishocolim}, the $K$-space $\mathcal P$ is naturally identified with the Grothendieck construction
\[
\mathcal P\simeq \operatorname{Tr}(F),
\]
where $F\colon \mathcal Q \rightarrow \Cat$ is the diagram sending
$q$ to the coset space $K/G(q)$.
After geometric realization, Thomason's theorem \ref{simeq:hocolimTr} yields
\[
|\mathcal P|
\simeq
\operatorname{hocolim}_{Q^{op}} K/G(q).
\]

Applying the Borel construction and using that it can itself be expressed as a homotopy colimit (see Subsection \ref{prelimhomotopy}), together with the fact that homotopy colimits commute (see \cite[Theorem 24.9]{CS}), we obtain
\[
|\mathcal P|_{hK}
\simeq
\operatorname{hocolim}_{Q^{op}}
\bigl(EK\times_K K/G(q)\bigr).
\]

Finally, for every $q\in Q$, there is a natural homeomorphism
\[
EK\times_K K/G(q)
\cong
EK/G(q),
\]
and since $EK$ is a free contractible $G(q)$-space, we have
\[
EK/G(q)\simeq BG(q).
\]
Therefore
\[
|\mathcal P|_{hK}
\simeq
\operatorname{hocolim}_{Q^{op}} BG(q),
\]
which is the desired description.
\end{proof}
\begin{Obs}
    Note that the homotopy type of $|\mathcal P|_{hK}$ does not depend on $\varphi$. In fact, if there exists $\varphi\colon \hat{G}\rightarrow K$ a group homomorphism such that the restriction to $G(q)$ is injective for each $q\in \mathcal Q$, then $id\colon \hat{G}\rightarrow \hat{G}$ also satisfies them. In fact, if we consider the category $\textrm{Rep}(G)$ of pairs $(\varphi,K)$ where $\varphi\colon \hat{G}\rightarrow K$ which satisfy previous conditions, and morphisms from $(\varphi,K)$ to $(\varphi',K')$ are given by $f\colon K\rightarrow K'$ such that $f\circ \varphi=\varphi'$, then it is either empty or it has an initial object $(id,\hat{G})$.
\end{Obs}
We will describe some trivial examples.
\begin{Ex}
Let $\CQ$ be the trivial poset $\CQ=\{1\}$, and let $G\colon \{1\}\rightarrow \textrm{InjGroups}$ be a simple complex of groups such that $G(1)=G$ a finite group. Then, $\hat{G}=G$ and let $\varphi\colon G\to K$ be a group homomorphism satisfying the conditions of \ref{def:associatedGspace}, we compute $\CP=K/G$ and in this case $|\CP|_{hK}=|K/G|_{hK}\simeq BG$.

Thus, let $G$ be the trivial group $G=e$, then $\CP=K$ and $|\CP|_{hK}=\ast$ is contractible.
\end{Ex}

\begin{Ex}
On the other hand, let $\CQ$ be a connected poset, and let $G\colon \CQ\rightarrow \textrm{InjGroups}$ be a simple complex of groups such that $G(q)=e$ the trivial group for every $q\in \CQ$. In this case, $\hat{G}=e$ and we compute $\CP=K\times \CQ=\sqcup_{\CQ}K$, so $|\CP|_{hK}\simeq \hocolim_\CQ Be\simeq |\CQ|.$
\end{Ex}

\begin{Ex}
    We describe in detail the case of the symmetric groups. Let $r_i=(i,i+1)$, then the symmetric group can be presented as a symmetric group $$\mathcal{S}_n=\langle r_1,\dots r_{n-1}|r_i^2, (r_ir_{i+1})^3, (r_ir_j)^2 \textrm{ if } |i-j|>1 \rangle .$$ Then we denote by $$X:=\{1,\dots,n-1\},$$ and let $A$ be any subset of $X$ such that $A\subsetneq X$. Now we can define a simple complex of groups:
\begin{align*}
    G\colon \mathcal{Q}_X \longrightarrow & \textrm{InjGroups}\\
    A\mapsto & <r_i\mid i \in A>\\
    \emptyset \mapsto & e
\end{align*}
where $\mathcal{Q}_X$ is the poset of the proper subsets of $X$ with the set inclusion as the ordering relation.

For example, if we assume that \underline{$n=3$}:

We will have
\begin{align*}
    G\colon \mathcal{Q}_{\{1,2\}} \longrightarrow & \textrm{InjGroups}\\
    \emptyset \mapsto & e\\
    \{1\}\mapsto & <r_1>\cong\mathbb{Z}/2\\
    \{2\}\mapsto & <r_2>\cong\mathbb{Z}/2
\end{align*}

Note that $$<r_1,r_2 \mid r_1^2=r_2^2=1, r_1r_2r_1=r_2r_1r_2>.$$ Thus, we can define the poset $$\mathcal{P}(\mathcal{Q}_{\{1,2\}}, G)=\mathcal{S}_3/e\bigsqcup \mathcal{S}_3/<r_1>\bigsqcup\mathcal{S}_3/<r_2>,$$ with the partial ordering $$(\emptyset,k\cdot e)<(i,k'<r_i>)$$ iff $\emptyset<i$ that is true for any $i\in \{1,2\}$, and $k^{-1}k'\in<r_i>$. 

We need to remark that the set of elements of the group $\mathcal{S}_3$ is $$\{e,r_1,r_2r_1,r_1r_2r_1, r_2, r_1r_2\}.$$ Then the set of elements of the group $\mathcal{S}_3/<r_1>$ is $\{[e],[r_2], [r_1r_2]\}$. Similarly, the set of elements of the group $\mathcal{S}_3/<r_2>$ is $\{[e],[r_1], [r_2r_1]\}$.

Let's assume that $i=1$ (the other case is symmetric). Then, $k'\in\{e, r_2, r_1r_2\}$, and therefore $k^{-1}k'\in<r_1>$ if and only if $k^{-1}k'=e$ or $k^{-1}k'=r_1$. 
\begin{itemize}
    \item If $k'=e$, then $k=e$ or $k=r_1$
    \item If $k'=r_2$, then $k=r_2r_1$ or $k=r_2$
    \item if $k'=r_1r_2$, then $k=r_1r_2r_1$ or $k=r_1r_2$.
\end{itemize}

Thus, we can represent the $K$-poset $\mathcal{P}(\mathcal{Q}_{\{1,2\}}, G)$:

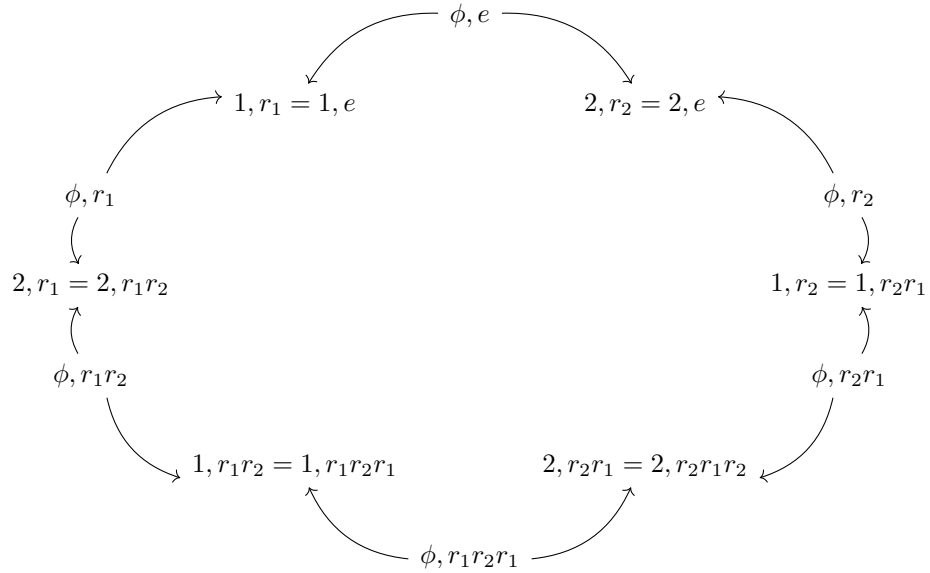
\begin{tikzcd}[column sep=0]
                                                           &                        & {\phi, e} \arrow[ld, bend right] \arrow[rd, bend left]         &                        &                                                            \\
                                                           & {1, r_1=1,e}           &                                                                & {2, r_2=2,e}           &                                                            \\
{\phi, r_1} \arrow[d, bend right] \arrow[ru, bend left]    &                        &                                                                &                        & {\phi, r_2} \arrow[d, bend left] \arrow[lu, bend right]    \\
{2, r_1 = 2, r_1r_2}                                       &                        &                                                                &                        & {1,r_2=1,r_2r_1}                                           \\
{\phi, r_1r_2} \arrow[rd, bend right] \arrow[u, bend left] &                        &                                                                &                        & {\phi, r_2r_1} \arrow[ld, bend left] \arrow[u, bend right] \\
                                                           & {1,r_1r_2=1,r_1r_2r_1} &                                                                & {2,r_2r_1=2,r_2r_1r_2} &                                                            \\
                                                           &                        & {\phi, r_1r_2r_1} \arrow[ru, bend right] \arrow[lu, bend left] &                        &                                                           
\end{tikzcd}

We can also compute $$<r_1>\backslash\mathcal{S}_3/<r_1>=\{[e],[r_2]\}$$ and $$<r_2>\backslash\mathcal{S}_3/<r_1>=\{[e],[r_1r_2]\}.$$ And similarly, $$<r_2>\backslash\mathcal{S}_3/<r_2>=\{[e],[r_1]\}$$ and $$<r_1>\backslash\mathcal{S}_3/<r_2>=\{[e],[r_2r_1]\}.$$
\end{Ex}

\begin{Ex}
Let $\mathcal A_p(G)$ be the poset of elementary abelian $p$-subgroups of a finite group $G$ and $sd\mathcal A_p(G)$ its barycentric subdivision, the poset of chains. The group $G$ acts by conjugation on $sd\mathcal A_p(G)$, let $[sd\mathcal A_p(G)]$ be the poset of conjugacy classes of chains. 

In this context, one can define a simple complex of groups on the poset $[sd\mathcal A_p(G)]$ as follows. Let $$N\colon [sd\mathcal A_p(G)]\rightarrow Groups$$ be the functor defined by $$N([E_0<\ldots<E_n])=\bigcap N_G(E_i).$$

\begin{enumerate}
    \item The associated poset as in Definition \ref{def:associatedGspace} is equivalent to $sd\mathcal A_p(G)$ (which has the same geometric realization as $\mathcal A_p(G)$).
    \item If $|\mathcal A_p(G)|$ is $1$-connected then the finite group $G$ is the associated group of the simple complex of groups $N$, and $|\mathcal A_p(G)|$ is the universal cover of $$\hocolim_\mathcal{Q} EG\times_G G/G(q).$$
    \item If $|\mathcal A_p(G)|\simeq \ast$ then the finite group $G$ is the associated group of the simple complex of groups $N$, and the higher homotopy groups are trivial.
    
\end{enumerate}

One motivation for the study of simple complexes of groups is the way they arise in relation with decompositions of finite groups.

\end{Ex}

\subsection{The join of simple complexes of groups}\label{sec:join}

The geometric join construction of spaces has the advantage of increasing connectivity. This property is used to construct universal spaces for actions with prescribed isotropy. For example $EG$ has a model $\bigcup_n *_n G$ which is a contractible space with a free $G$-action. 
The goal of this section is to see if there is an analogue of the geometric join, but at the level of simple complexes of groups that is compatible with the geometric realization. 

\begin{Def}\label{r:join}
    Given two posets $\mathcal{P}$ and $\mathcal{Q}$, we define $\mathcal{P}*\mathcal{Q}$ as the new poset with objects  $\mathcal{P}\times \mathcal{Q} \cup \mathcal P \cup \mathcal Q$ with a new type of relation $p >(p,q)<q$ for any $p\in \mathcal{P}$ and $q\in \mathcal{Q}$. For the rest are the usual objectwise relations. 
\end{Def}

\begin{Lema}
    Let $\mathcal{P}$ and $\mathcal Q$ be two posets. Then $|\mathcal{P}*\mathcal{Q}|\simeq |\mathcal{P}|*|\mathcal{Q}|$.
\end{Lema}

\begin{proof}
The proof is based on Thomason's theorem \ref{simeq:hocolimTr}, that is, on the description of the homotopy colimit of a functor to nerves of categories as the nerve of the transporter category; and the description of the join of two spaces $X$ and $Y$ as the homotopy colimit of the diagram $X\leftarrow X\times Y \rightarrow Y$ where the maps are the corresponding projections.

Let
\[
I=\bigl(1 \longleftarrow 0 \longrightarrow 2\bigr)
\]
be an indexing category, and define a functor $F\colon I\longrightarrow \Cat$ by
\[
F(1)=\mathcal P,\qquad
F(0)=\mathcal P\times\mathcal Q,\qquad
F(2)=\mathcal Q,
\]
where the morphisms are the canonical projections
\[
\pi_{\mathcal P}\colon\mathcal P\times\mathcal Q\longrightarrow\mathcal P,
\qquad
\pi_{\mathcal Q}\colon\mathcal P\times\mathcal Q\longrightarrow\mathcal Q.
\]

Recall that the Grothendieck construction \(Tr(F)\) has objects $\bigsqcup_{i\in I}F(i),$ which in this case are described by $\mathcal P\sqcup(\mathcal P\times\mathcal Q)\sqcup\mathcal Q$, and its order is generated by the orders on each component together with
\[
p>(p,q)<q,
\]
for every \(p\in\mathcal P\) and \(q\in\mathcal Q\), coming from the projection maps. Hence $
Tr(F)=\mathcal P*\mathcal Q
$

Applying Thomason's homotopy colimit Theorem \ref{simeq:hocolimTr} yields a natural homotopy equivalence
\[
|Tr(F)|
\simeq
\operatorname{hocolim}|F|.
\]
Since \(Tr(F)=\mathcal P*\mathcal Q\), we obtain
\[
|\mathcal P*\mathcal Q|
\simeq
\operatorname{hocolim}
\left(
|\mathcal P|
\longleftarrow
|\mathcal P\times\mathcal Q|
\longrightarrow
|\mathcal Q|
\right).
\]

Finally, because geometric realization preserves products of simplicial sets, $|\mathcal P\times\mathcal Q|
\cong
|\mathcal P|\times|\mathcal Q|$ (see \cite[Theorem 14.3]{May67}),  and the homotopy colimit of the diagram
\[
|\mathcal P|
\longleftarrow
|\mathcal P|\times|\mathcal Q|
\longrightarrow
|\mathcal Q|
\]
is the join $|\mathcal P|*|\mathcal Q|$. Therefore,
$
|\mathcal P*\mathcal Q|
\simeq
|\mathcal P|*|\mathcal Q|
$.
\end{proof}

\begin{Obs}
    One the main properties of the join construction is that increases the connectivity: if $X$ is $n$-connected and $Y$ is $m$-connected then $X*Y$ is $n+m+1$-connected. In particular $*_{\infty} X=\cup_n *_N X$ is weakly contractible.
\end{Obs}

We concentrate on the case of $|\mathcal{P}|*|\mathcal{P}|$ where $|\mathcal{P}|$ is as in Definition \ref{d:simplecomplexgroups}. 

\begin{Obs}\label{def:joinofsystems}
    Let $\varphi\colon \hat{G}\rightarrow K$ be a group homomorphism such that the restriction to $G(q)$ is injective for each $q\in \mathcal Q$. We denote by $\mathcal P$ the associated $K$-poset. Following Definition \ref{r:join} the poset $\mathcal{P}*\mathcal{P}$ can be described as follows. The elements are
    \begin{itemize}
        \item $((q,kG(q)),(q',k'G(q')))$, where $q,q'\in\mathcal{Q}$ and $k,k'\in K$,
        \item $((q,kG(q)),0)$ with $q\in\mathcal{Q}$ and $k,\in K$,
        \item $(0,(q',k'G(q')))$ with $q'\in\mathcal{Q}$ and $k',\in K$,
    \end{itemize} 
    and poset relations generated by
    \begin{align*}
    ((q_1,k_1G(q_1)),0)<((q_2,k_2G(q_2)),0) &\Longleftrightarrow (q_1,k_1G(q_1))<(q_2,k_2G(q_2));\\
    ((q_1,kG(q_1)),0)>((q_2,k_2G(q_2)),(q_3,k_3G(q_3))) &\Longleftrightarrow (q_1,k_1G(q_1))>(q_2,k_2G(q_2));\\
    (0,(q_1,k_1G(q_1)))<(0,(q_2,k_2G(q_2))) &\Longleftrightarrow (q_1,k_1G(q_1))<(q_2,k_2G(q_2));\\
    (0,(q_1,k_1G(q_1)))>((q_2,k_2G(q_2)),(q_3,k_3G(q_3))) &\Longleftrightarrow (q_1,k_1G(q_1))>(q_3,k_3G(q_3));\\
    \end{align*} 
    and, $$((q_1,k_1G(q_1)),(q_2,k_2G(q_2)))<((q_3,k_3G(q_3)),(q_4,k_4G(q_4)))$$ if and only if $((q_1,k_1G(q_1))<((q_3,k_3G(q_3))$ and $(q_2,k_2G(q_2))<(q_4,k_4G(q_4)).$
\end{Obs}

\begin{Obs}
    Let $G$ be a discrete group, and let $H$ and $K$ be subgroups, there is a classical formula describing a decomposition of the product of orbits  $G/H\times G/K=\bigsqcup_{[g]\in H\backslash G/K}G/H\cap {}^{g}K$, which applied to the particular case gives $$K/G(q)\times K/G(q')=\bigsqcup_{\alpha_i\in G(q)\backslash K/(G(q')}K/G(q)\cap {}^{\alpha_i}G(q').$$
\end{Obs}

The goal is to describe the join in Definition \ref{def:joinofsystems} as a new simple complex of groups and to identify the corresponding structure.

We denote $\mathcal{Q}_0:=\mathcal{Q}$, and then $G_0:=G$ and $\mathcal{P}_0:=\mathcal{P}(\mathcal{Q}_0,G_0)$. Now we will define $\mathcal{Q}_1$ and $G_1$ such that $\mathcal{P}_1:=\mathcal{P}(\mathcal{Q}_1, G_1)\cong\mathcal{P}_0 * \mathcal{P}_0$. 

Let $\mathcal{Q}_1$ be the subposet of $\mathcal{P}_0*\mathcal{P}_0$ whose objects are of the form $((q,eG(q)),0)$, $(0,(q',\alpha_iG(q')))$ and $((q,eG(q)),(q',\alpha_iG(q')))$ where $\alpha_i$ are representatives of the double coset $G(q)\backslash K/G(q')$.

Then we are going to define a simple complex of groups on $\mathcal{Q}_1$ 

\begin{Prop}
    Let $\mathcal{Q}_1$ as before. Then \begin{align*}
    G_1\colon\mathcal{Q}_1\longrightarrow & \textrm{InjGroup}\\
    (q,0)\mapsto &G(q)\\
    (0,q',\alpha_iG(q'))\mapsto & {}^{\alpha_i}G(q')\\
    (q,q',\alpha_iG(q'))\mapsto &G(q)\cap ^{\alpha_i}G(q').
\end{align*}
is a functor such that $(\mathcal{Q}_1,G_1)$ is a simple complex of groups with a monomorphism $\varphi:\hat{G_1}\rightarrow K$ and associated $K$-poset $\mathcal{P}_1$ with $|\mathcal{P}_1|\simeq |\mathcal{P}_0|*|\mathcal P_0|$. 
\end{Prop}

\begin{proof}
$G_1$ is a functor, because all the relations in $\CQ_1$ induce inclusion relations in their images by $G_1$. Then, $\hat{G}_1\colon=\colim_{\mathcal{Q}_1}G_1$, and we get $\varphi^*\colon \hat{G_1}\rightarrow K$ such that its associated poset $\mathcal{P}_1$ satisfies $|\mathcal{P}_1|\simeq\hocolim_{\mathcal{Q}_1}|F|$ where $|\mathcal{P}_1|\simeq|Tr(\mathcal{Q}_1,F)|$. 

Now, to prove that $|\mathcal{P}_1|\simeq|\mathcal{P}_0|*|\mathcal{P}_0|$, we show that both categories are equivalent by describing functors in both directions:

\begin{align*}
    V\colon\mathcal{P}(\mathcal{Q}_1, G_1)\longrightarrow &\mathcal{P}(\mathcal{Q}_0, G_0)*\mathcal{P}(\mathcal{Q}_0, G_0)\\
    ((q,q',\alpha_iG(q')),[{k}]\in K/G(q)\cap ^{\alpha_i}G(q')) \mapsto &((q,{k}G(q)),(q',{k}\alpha_iG(q')))\\
    ((q,0),[k]\in K/G(q))\mapsto &((q,kG(q)),0)\\
    ((0,q,\alpha_iG(q)),[{k}]\in K/^{\alpha_i}G(q))\mapsto &(0,(q,{k}\alpha_iG(q))).
\end{align*}

We also have the functor:
\begin{align*}
    W\colon\mathcal{P}(\mathcal{Q}_0, G_0)*\mathcal{P}(\mathcal{Q}_0, G_0) \longrightarrow & \mathcal{P}(\mathcal{Q}_1, G_1)\\
    ((q,kG(q)),0) \mapsto & ((q,0),[k]\in K/G(q))\\
    (0,(q,{k}G(q)))  \mapsto & ((0,q,kG(q)),[e]\in K/{}^kG(q))\\
    ((q,kG(q)),(q',k'G(q'))) \mapsto & ((q,q',k^{-1}k'G(q')),[k]\in K/G(q)\cap ^{k^{-1}k'}G(q')).
\end{align*}

\end{proof}

We can iterate this process to obtain new simple complexes of groups $(\mathcal{Q}_n,G_n)$ with morphisms $\varphi_n\colon \hat{G}_n\to K$, such that the corresponding associated $K$-poset $\mathcal{P}_n$ satisfies $$|\mathcal{P}_n|\simeq |\mathcal P_{n-1}|*|\mathcal P|\simeq *_n|\mathcal P|$$ with the property that the values of $G_n\colon \mathcal{Q}_n\to \textrm{InjGroups}$ are of the form $$G(q)\cap { }^{k_1}G(q_1) \cap \dots \cap { }^{k_{n-1}}G(q_{n-1}).$$ 

\begin{Obs}
    Note that at each step the family of subgroups increases by taking conjugations and subgroups. In the limit one obtain the family of subgroups of $K$ generated by $\{G(q)\}_{q\in \mathcal{Q}}$, that is, those subgroups obtained by intersections and conjugations. 
\end{Obs}

Even if $\mathcal{Q}$ is finite, this property is no longer satisfied when we iterate this process. Next step $\mathcal{Q}_1$ is no longer finite if $K$ is not, but the length of the poset (Definition \ref{lengposet}) is controlled. 

\begin{Lema}\label{lengjoin}
    Let $\mathcal{P}$ and $\mathcal{Q}$ be posets of finite length $n$ and $m$, then $\leng(\mathcal{P}*\mathcal{Q})=n+m+1.$
\end{Lema}

\begin{proof}
By the hypothesis of the lemma, we have a maximal chain in $\CP$ $p_1<\dots<p_{n+1}$ and a maximal chain in $\CQ$ $q_1<\dots<q_{m+1}$. Then, a maximal chain in $\mathcal{P}*\mathcal{Q}$ is $$(p_1,q_1)<\dots<(p_1,q_{m+1})<p_1<\dots<p_{n+1},$$ so $\leng(\mathcal{P}*\mathcal{Q})=n+m+1.$
\end{proof}

\begin{Coro}
    Let $(\mathcal{Q}, G)$ be a simple complex of groups with $\varphi\colon \hat{G}\to K$ and associated poset $\mathcal{\mathcal P}$, then $\leng(\mathcal Q_1)\leq 2\leng(\CP)+1.$
\end{Coro}

\begin{proof}
By definition, $\mathcal{Q}_1$ is a subposet of $\mathcal{P}*\mathcal{P}$. Hence, $\leng(\mathcal Q_1)\leq\leng(\mathcal{P}*\mathcal{P})$. By Lemma \ref{lengjoin}, $\leng(\mathcal{P}*\mathcal{P})=\leng(\CP)+\leng(\CP)+1,$ and the result follows.
\end{proof}

\section{Mod $p$ and $p$-adic cohomology theory of simple complexes of groups}\label{C:ModpCohomology}\hfill

The mod $p$ cohomology of Coxeter groups (not necessarily finite) enjoys finiteness properties. A property of these groups is that they have a reflection representation, which gives a faithful complex representation $W\hookrightarrow U(n)$. The argument used for finite groups to show that their mod $p$ cohomology is Noetherian also applies in this case. %

This section closely follows the strategy used by Kitchloo and Broto in \cite{BK:Kac-Moody_groups} in their study of the mod $p$ cohomology of classifying spaces of Kac-Moody groups.

We now state the main result of this section. It establishes a finite generation theorem for the mod $p$ cohomology of the Borel construction associated to a simple complex of groups and its induced $K$-action on the corresponding poset.

\begin{Teo}\label{main}
    Let $G$ be a simple complex of groups over a finite poset, with an homomorphism $\varphi \colon \hat{G}\to K$ and an associated $K$-poset $\mathcal P$ defined as in Definition \ref{def:associatedGspace}. The mod $p$ cohomology $H^*(|\mathcal P|_{hK}; \mathbb{F}_p)$ is a finitely generated $\mathbb{F}_p$-algebra.
\end{Teo}

\begin{Obs}\label{Noetherian-fingen}
Let \(A = \bigoplus_{n \geq 0} A_n\) be an \(\mathbb{N}\)-graded ring. Then \(A\) is Noetherian if and only if \(A_0\) is Noetherian and \(A\) is finitely generated as a ring over \(A_0\) \cite[Theorem~13.1]{Matsumura89}. In particular, let $R$ be either $\hat{\mathbb{Z}}_p$ or $\mathbb{F}_p$, which are Noetherian rings. The cohomology $H^*(Y;R)$ of any connected space $Y$ will be a graded algebra. Thus, the previous Theorem applies, and $H^*(Y;R)$ is a Noetherian $R$-algebra if and only if it is finitely generated as an $R$-algebra. It will be important to have this in consideration for the rest of the paper.
\end{Obs}

\begin{Not}
As in the previous sections, we write either $H^*X$ or $H^*(X)$ for the mod $p$ cohomology $H^*(X;\mathbb{F}_p)$ of a space $X$. Note also that $V$ will denote an elementary abelian $p$-group.
\end{Not}

To prove Theorem \ref{main}, we will use the results of the previous section (Theorem \ref{mainHomotopy}), and we will show that $H^*(\hocolim_{\mathcal{Q}}  BG(q))$ is Noetherian. We will make use of the Bousfield-Kan spectral sequence, which is a multiplicative spectral sequence 
\begin{equation}
    \lim_{q\in \mathcal{Q}}{ }^i H^j(BG(q))\Rightarrow H^{i+j}(|\mathcal P|_{hK}).
\end{equation}
We describe the different steps we will follow to achieve the conclusion:
\begin{enumerate}
    \item\label{F} The edge homomorphism in the Bousfield-Kan spectral sequence $$ H^*(|\mathcal P|_{hK})\to  \lim_{q\in \mathcal{Q}} H^*(BG(q))$$ is an $F$-isomorphism.
    \item\label{S.three} The $\textrm{End}(V)$-set $s_d(H^*(|\mathcal P|_{hK}))$ is Noetherian for any $d$.
    \item\label{point}  There exists a finitely generated sub-$\mathcal{A}_p$-algebra $C\subseteq H^* (|\mathcal P|_{hK})$ with the following property: for every $q\in\mathcal{Q}$, $H^*(BG(q))$ is a finite $C$-module via the map $C\subseteq H^* (|\mathcal P|_{hK})\to H^*(BG(q))$.
\end{enumerate}

\subsection{The $F$-isomorphism}\label{sec:Fiso}

The algebraic general situation that will produce the $F$-iso\-mor\-phism we are interested in Step \ref{F} is standard.

Let $\{E_2^{*,*}\}$ be a first quadrant multiplicative cohomological spectral se\-quen\-ce of $\mathbb{F}_p$-modules converging to $H^*$. That is, $E_\infty$ recovers the graded module associated to a filtration of $H^*$. In particular, there is an edge homomorphism $$\rho \colon H^*\rightarrow E^{0,*}_\infty\subset E^{0,*}_2.$$

The following proposition is well-known.

\begin{Prop}\label{prop:Fiso_SS}
    Let $\{E_2^{*,*}\}$ be a first quadrant multiplicative cohomological spec\-tral sequence of $\mathbb{F}_p$-algebras converging to $H^*$. If there is $N>0$ such that $E_2^{p,q}=0$ for $p>N$, then the edge homomorphism $\rho$ is an $F$-isomorphism.
\end{Prop}

\begin{Obs}
    Note the previous result Proposition \ref{prop:Fiso_SS} is also valid if we replace $\mathbb{F}_p$ by a field $k$ of characteristic $p$.
\end{Obs}

\begin{Obs}\label{remark:Niso}
    If $E_2^{0,*}$ has an exponent, that is, there is an $K>0$ such that $K\cdot E_2^{i,n}=0$ for all $n>0$, all $i$. Then the proof adapts to show that the edge morphism has the property that if $y\in \ker\rho$ then there is $L\geq 0$ such that $y^{K^L}=0$.
\end{Obs}

\begin{Coro}
    Under the hypothesis of Proposition \ref{prop:Fiso_SS}, $H^*$ and $E^{0,*}_2$ have the same Krull dimensions and transcendence degree.
\end{Coro}
    
We will apply this result to the Leray spectral sequence associated to the map $\hocolim_{\mathcal{Q}^{op}} EK\times_K F(q)\rightarrow |\mathcal{Q}|$. In this spectral sequence, the $E_2$-term is the cohomology of $|Q|$ with twisted coefficients, $H^p(\mathcal Q; \mathcal H^q)$. In particular, $H^0(\mathcal Q; \mathcal H^q)=\lim_{{\mathcal Q}} \hbgi$. This is also known as the Bousfield-Kan spectral sequence. 

Recall that $|\mathcal{Q}|$ is a finite poset, therefore it has a finite filtration by skeleta, and the corresponding spectral sequence for mod $p$ cohomology satisfies the hypothesis of Proposition \ref{prop:Fiso_SS}. Then higher limits of a contravariant functor $\lim_\mathcal{Q}^i F$ are com\-put\-ed as the homology of a complex
\begin{align*}\label{diag_resolutionlim}
\prod_{q_i} F(q_i)\leftarrow \prod_{q_i\to q_j} F(q_j) \leftarrow \prod_{q_i\to q_j \to q_k} F(q_k)\leftarrow \cdots
\end{align*} That is, we have the following lemma.

\begin{Lema}\label{lem:bounded limits}
    Let $\mathcal{Q}$ be a poset of finite length and $F\colon \mathcal Q^{op} \rightarrow R-\Mod$ where $R$ is a commutative ring. Then there exists $N>0$ such that $lim {}^{i}F=0$ for all $i>N$. 
\end{Lema}

\begin{proof}
    Set $N=\leng (\CQ)+1$.
\end{proof}

\begin{Obs}
   If $\mathcal{Q}$ is finite and $F(q)$ are finitely generated $R$-modules with $R$ Noe\-the\-ri\-an, then we can say more: $\lim {}^{i}F=0$ are finitely generated $R$-modules.
\end{Obs}

We are now in a position to establish Step \ref{F} of the proof of Theorem \ref{main}.

\begin{Prop}\label{prop:Fiso_Q}
    There is an $F$-isomorphism 
    \begin{equation}
        H^* (|\mathcal P|_{hK};k)\xrightarrow{\rho} \lim_{{\mathcal Q}} H^* (BG(q);k)
    \end{equation}
    where $k$ is a field of characteristic $p$.
\end{Prop}

\begin{proof}
We use the description as a homotopy colimit. There is a first quadrant multiplicative spectral sequence with $E_2$-term given by 
$$E_2^{i,j}=\lim {}^{i}H^j(BG(q);k)$$
converging to $H^*( |\mathcal P|_{hK};k)$. Since the indexing category is a finite poset, we apply Lemma \ref{lem:bounded limits} to obtain that there is $N\geq 0$ such that $\lim {}^{i}H^j(BG(q);\mathbb F_p)=0$ for every $j$ and for all $i>N.$ In this case, the edge homomorphism 
coincides with the one induced by restricting along subgroups and $$H^*( |\mathcal P|_{hK};k) \xrightarrow{\rho} \lim_{{\mathcal Q}} H^*(BG(q);k)$$ is an $F$-isomorphism by Proposition \ref{prop:Fiso_SS}.  
\end{proof}

One direct consequence of the previous $F$-isomorphism theorem is the descrip\-tion of the spectrum of homogeneous prime ideals and the finiteness of the Krull dimension and of the transcendence degree.

\begin{Coro}
    Let $G$ be a simple complex groups over a finite poset, with $\varphi \colon \hat{G}\to K$ and an associated $K$-poset $\mathcal P$ defined as in Definition \ref{def:associatedGspace}. There is a homeomorphism 
    $$\colim_\mathcal{Q} \textrm{Spec}^h(H^*(BG(q)))\cong \textrm{Spec}^h(H^*(|\mathcal P|_{hK}))$$ and the Krull dimension of  $H^*(|\mathcal P|_{hK})$ is finite.    
\end{Coro}

\begin{Coro}\label{cor:tdfinite}
     Let $G$ be a simple complex groups over a finite poset, with $\varphi \colon \hat{G}\to K$ and an associated $K$-poset $\mathcal P$ defined as in Definition \ref{def:associatedGspace}. The tran\-scend\-ence degree $d(H^*(|\mathcal P|_{hK}))$ is finite.
\end{Coro}

\begin{proof}
    We have that $d(H^*(|\mathcal P|_{hK}; \mathbb{F}_p))=d(\lim_{{\mathcal Q}} H^*(BG(q);k))$. Note that we have an inclusion into a finite product $\lim_{{\mathcal Q}} H^*(BG(q);k)\subseteq \prod_{\mathcal{Q}}H^*(BG(q);k)$, then we are reduced to prove the statement for $H^*(BG(q);k)$. Since $H^*(BG(q);k)$ is Noetherian, for each $q\in \mathcal{Q}$ we have that $d(H^*(BG(q);k))$ is finite and equal to the Krull dimension.
\end{proof}

The next goal is to study finiteness properties, such as being a Noetherian algebra. The strategy is to combine the $F$-isomorphism statement with \cite{BK:Kac-Moody_groups} and \cite[Proposition 5.2]{BLO:homotopy_fusion_systems}. We use the technology developed in \cite{HLSmodulonilpotents}.

We proceed to study $\textrm{End}(V)$-sets associated with $H^*(|\mathcal{P}|_{hK})$. The following proposition corresponds to Step \ref{S.three} of the proof of Theorem \ref{main}.

\begin{Prop}\label{PropStep3}
    Let $K=H^*(|\mathcal P|_{hK})$. For each $d$, there is an isomorphism of $\textrm{End}(V)$-sets, $$s_d(K)\cong \textrm{colim}_{\mathcal Q} \Rep(V,G(q)),$$ and $s_d(K)$ is Noetherian.
\end{Prop}

\begin{proof}
We use the $F$-isomorphism established in Proposition \ref{prop:Fiso_Q}.
    \begin{equation*}
        H^*(|\mathcal P|_{hK}) \xrightarrow{} \lim_{{\mathcal Q}} H^*(BG(q))
    \end{equation*}
to conclude that $$s_d(K)\cong s_d(\lim_{{\mathcal Q}} H^*(BG(q))).$$ Moreover, $$s_d(\lim_{{\mathcal Q}} H^*(BG(q)))=\colim_{{\mathcal Q}} s_d(H^*(BG(q)))\cong \colim_{{\mathcal Q}} \Rep(V,G(q)).$$ The last equivalence follows from Lannes theory \cite{lannes1987cohomologie}, see \cite[Section 9.6]{Schwartz_1994}. So far, we have found that $s_d(K)\cong \colim_{{\mathcal Q}} \Rep(V,G(q)).$

Finally, $\Rep(V,G(q))$ is a Noetherian $\End(V)$-set for every $q\in \CQ$, and for every $i\colon q\rightarrow q'$, we have that $$\textrm{Rep}(V,G(i))\colon \Rep(V,G(q)) \to \Rep(V,G(q'))$$ is a morphism of $\End(V)$-sets that preserves kernels. So, as all the conditions of Proposition \ref{p:colimNoetherian} are satisfied, we can use it to affirm that $\colim_{{\mathcal Q}} \textrm{Rep}(V,G(q))$ is a Noetherian $\textrm{End}(V)$-set, that is, $s_d(K)$ is a Noetherian $\End(V)$-set.
\end{proof}

\subsection{Finiteness properties of mod $p$ cohomology via $F$-isomorphism}

In this setting, we have the following definition introduced in \cite{Powell}.

\begin{Def}
   An unstable algebra $K$ is Noetherian up to nilpotents if it has finite transcendence degree and, for any $d\in \mathbb{N}$, $s_d(K)$ is a Noetherian $\End(V)$-set, where $V=(\mathbb{Z}/p)^d$.
\end{Def}

\begin{Lema}
    Let $K$ be an unstable algebra which is Noetherian up to nilpotents. There is an $F$-isomorphism $f\colon K\rightarrow L$ where $L$ is unstable Noetherian $\mathbb{F}_p$-algebra.
\end{Lema}

\begin{proof}
    Let $d$ be the transcendence degree of $K$. Since $s_d(K)$ is a Noetherian $\End(V)$-set, we have that $b_d(s_d(K))$ is an unstable Noetherian $\mathbb{F}_p$-algebra. We take $f\colon K\to b_d(s_d(K))$, which is an $F$-isomorphism by Theorem \ref{bdsd:fiso}.
\end{proof}

Recall that by Proposition \ref{PropStep3}, $s_d(H^*(|\mathcal P|_{hK}))$ is a Noetherian $\End(V)$-set for any $d$ and moreover the transcendence degree is finite (Corollary \ref{cor:tdfinite}).

\begin{Coro}\label{Obs.Borel_fgNil}
    Let $G$ be a simple complex of groups over a finite poset, with an homomorphism $\varphi \colon \hat{G}\to K$ and an associated $K$-poset $\mathcal P$ defined as in Definition \ref{def:associatedGspace}. The mod $p$ cohomology $H^*(|\mathcal P|_{hK})$ is Noetherian up to nilpotents.
\end{Coro}

\begin{Prop}\label{general}
Let $A\in \mathcal{K}$ be an unstable Noetherian $\mathbb{F}_p$-algebra up to nilpotents. Let $B$ be an $\mathcal{A}_p$-algebra with $B$ a finitely generated $\mathbb{F}_p$-algebra, and $f\colon A\to B$ be an $F$-isomorphism. There exists $C\subseteq A$ a finitely generated sub-$\mathcal{A}_p$-algebra such that $$f\big|_{C}\colon C\to B$$ is an $F$-isomorphism.
\end{Prop}

\begin{proof}
Since $B$ is a finitely generated algebra, we can choose a finite set of generators $\{x_i\}$ with $1\leq i \leq r$. Since $f$ an $F$-epimorphism, for each $i=1,\dots,r$, there exists $y_i\in A$ such that $$f(y_i)=x_i^{p^k}$$ form some $k$. Note that $k$ may depend on each $x_i$, but as there are finitely many such $x_i$, we can choose the same $k$ for all of them taking powers of the corresponding $y_i$.

Assume $p=2$. Let $$C' = \langle y_i \mid i=1, \dots, r \rangle \leq A$$ generated by $y_i$ as a sub-algebra. It satisfies $$f \left( Sq^i y_i \right) = Sq^i \left( f(y_i) \right) = Sq^i \left( x_i^{2^k} \right)\in L= \langle x_i^{2^k} \mid i=1, \dots, r \rangle $$ which is a sub-$\mathcal{A}_2$-algebra. Analogously, if $p>2$, then $$f \left( \mathcal{P}^i y_i \right) = \mathcal{P}^i \left( f(y_i) \right) = \mathcal{P}^i \left( x_i^{p^k} \right)\in L= \langle x_i^{p^k} \mid i=1, \dots, r \rangle $$ which is a sub-$\mathcal{A}_p$-algebra. In any case, we had $L\subseteq Im(f)$.
    
To follow the argument, we can assume that $p=2$ and then we will see that the case $p>2$ is analogous. Then, for a large enough $k$, there exists $w \in C'$ such that $$f(w) = Sq^i \left( x_i^{2^k} \right) = f(Sq^i y_i).$$

By the Cartan formula, $Sq^i(x^2)=\sum_{k+l=i} Sq^k x\cdot Sq^l x$, and commutativity, for $i$ odd, we will have that $Sq^i(x^2)=0$, and for $i$ even, we will have that $Sq^i(x^2)=(Sq^{i/2}(x))^2$. Note that this is the situation in Proposition \ref{sq0}.

Then, we can see how $$Sq^{i\cdot 2^s} (y_i^{2^s})=(Sq^{i} y_i)^{2^s}=w^{2^s}\in \langle y_i^{2^s}\rangle.$$

Similarly, if $p>2$, we have that for a large enough $k$, there exists $w \in C'$ such that $$f(w) = \mathcal{P}^i \left( x_i^{p^k} \right) = f(\mathcal{P}^i y_i).$$
    
Then, $$w - \mathcal{P}^i y_i \in \ker f$$ (which is an $F$-monomorphism), and consequently $$(w - \mathcal{P}^i y_i)^{p^s}=0.$$ Therefore, as we are working in a $\mathbb{F}_p$-algebra, $w^{p^s}=(\mathcal{P}^i y_i)^{p^s}$.

Again, we are in the situation of Proposition \ref{sq0}, and we have $\mathcal P^i\mathcal P_0(x)=\mathcal P_0 \mathcal P^{i/p}(x)$ if $i$ is a multiple of $p$. 
Then, we can see how $$\mathcal{P}^{i\cdot p^s} (y_i^{p^s})=(\mathcal{P}^i y_i)^{p^s}=w^{p^s}\in \langle y_i^{p^s}\rangle.$$
    
Now, we can define $$C = \langle y_1^{p^k},\dots , y_r^{p^k} \rangle \leq A$$ as a finitely generated sub-$\mathcal{A}_p$-algebra, and we can see that the restriction $f \big|_{C}\colon C\to B$ is an $F$-isomorphism.

At first, we will prove that it is a $F$-monomorphism. Note that as $C\leq A$, for every $c\in C$ such that $c\in \ker f \big|_{C}$, in particular $c\in A$ and it's also true that $c\in \ker f$. So, as $f$ is a $F$-monomorphism, there exists $t\in \mathbb{N}$ such that $c^t=0$, and thus, $f \big|_{C}\colon C\to B$ is an $F$-monomorphism.

Moreover, we would like to affirm that for every $x_i\in B$, there exists a $y_i^{p^r}\in C$ such that $$f(y_i^{p^r})=x_i^{p^k}$$ for some $p^k$. For now, as we have that $f\colon A\to B$ is a $F$-epimorphism, we know that for every $x_i\in B$, there exists a $y_i\in A$ such that $$f(y_i)=x_i^{p^s}$$ for some $p^s$. But then, there exists a $y_i^{p^r}\in C$ such that $$f(y_i^{p^r})=f(y_i)^{p^r}=(x_i^{p^s})^{p^r}=x_i^{p^k}.$$ Thus, we get that $f \big|_{C}\colon C\to B$ is an $F$-epimorphism.

 \end{proof}

Finally, the following proposition corresponds to Step \ref{point} of the proof of Theorem \ref{main}.

\begin{Prop}\label{prop:finiteCmodule}
    Let $G$ be a finite complex of groups over a finite poset $\mathcal{Q}$. There is a finitely generated sub-$\mathcal{A}_p$-algebra $C\subseteq H^*(|\mathcal P|_{hK})$ such that for every $q\in\mathcal{Q}$, $H^*(BG(q))$ is a finitely generated $C$-module via the restriction $C\to H^*(BG(q))$.
\end{Prop}

\begin{proof}
Now, as we argued in Remark \ref{Obs.Borel_fgNil}, $H^*(|\mathcal P|_{hK})$ is a finitely generated $\mathbb{F}_p$-algebra up to nilpotent elements. Also, $b_d(s_d(H^*(|\mathcal P|_{hK})))$ is a finitely generated $\mathcal{A}_p$-algebra. And finally, by Theorem \ref{bdsd:fiso}, $$H^*(|\mathcal P|_{hK})\to b_d(s_d(H^*(|\mathcal P|_{hK})))$$ is an $F$-isomorphism. Hence, the hypotheses of Proposition \ref{general} are satisfied, so we can apply it to get $$C\subseteq H^* (|\mathcal P|_{hK})\to b_d(s_d( H^* |\mathcal P|_{hK})),$$ where $d$ is the transcendence degree of $H^* (|\mathcal P|_{hK})$.  We  have $C\subseteq H^* (|\mathcal P|_{hK})\xrightarrow[]{f} H^*(BG(q)),$ then $H^*(BG(q))$ is a $C$-module through the map $f$. The goal is to show that $H^*(BG(q))$ is a finite $C$-module.

From \cite{HLSmodulonilpotents}, it is enough to show that the induced map of $\textrm{End}(V)$-sets, $$\mathrm{Hom}_{\mathcal{K}}(H^*BG(q),H^*V)\to \mathrm{Hom}(C,H^*V)$$ preserves kernels where $V$ is an elementary abelian $p$-group of rank $d$.

Now, by Proposition \ref{prop:FisoEndViso}, the inclusion $C\subseteq H^* (|\mathcal P|_{hK})$ is an $F$-isomorphism, and then we have $\mathrm{Hom}_{\mathcal{K}}(H^* (|\mathcal P|_{hK}), H^* V) \to \mathrm{Hom}_{\mathcal{K}}(C, H^* V)$ is an isomorphism. Then, to prove that $\mathrm{Hom}_{\mathcal{K}}(H^*BG(q),H^*V)\to \mathrm{Hom}(C,H^*V)$ preserves kernels, we only need to prove that $\mathrm{Hom}_{\mathcal{K}}(H^*BG(q),H^*V)\xrightarrow{f^*} \mathrm{Hom}_{\mathcal{K}}(H^* (|\mathcal P|_{hK}), H^* V)$ preserves kernels.

By Lannes' theory, we have the following description 
\begin{equation}\label{primerISO}
\mathrm{Hom}_{\mathcal{K}}(H^*BG(q),H^*V)\cong \mathrm{Rep}(V,G(q)).
\end{equation}

Moreover, $H^*(|\mathcal P|_{hK})\to \lim_\mathcal{Q} H^*BG(q)$ is an $F$-isomorphism (see Proposition \ref{prop:Fiso_Q}). Then finally, we have an isomorphism of $\textrm{End}(V)$-sets
\begin{equation*}
\mathrm{Hom}_{\mathcal{K}}(C,H^*V)\cong\mathrm{Hom}_{\mathcal{K}}(H^*(|\mathcal P|_{hK}),H^*V)\cong\mathrm{Hom}_{\mathcal{K}}(\lim_\mathcal{Q} H^*BG(q),H^*V).
\end{equation*}
And $\mathrm{Hom}_{\mathcal{K}}(\lim_\mathcal{Q} H^*BG(q),H^*V)\cong\colim \mathrm{Rep}(V,G(q))$, so we get
\begin{equation}\label{segonISO}
    \mathrm{Hom}_{\mathcal{K}}(H^*(|\mathcal P|_{hK}),H^*V)\cong \colim \mathrm{Rep}(V,G(q)).
\end{equation}

Then, by Proposition \ref{p:colimNoetherian}, we have that $\mathrm{Rep}(V,G(q))\to \colim \mathrm{Rep}(V,G(q))$ preserves kernels. Therefore, using the isomorphisms \ref{primerISO} and \ref{segonISO}, we conclude that
$$\mathrm{Hom}_{\mathcal{K}}(H^*BG(q),H^*V)\xrightarrow{f^*} \mathrm{Hom}_{\mathcal{K}}(H^* (|\mathcal P|_{hK}), H^* V)$$ preserves kernels.
\end{proof}

Now, we have all the ingredients for the proof of the main theorem.

\begin{proof}[Proof of Theorem \ref{main}]
    With Proposition \ref{prop:finiteCmodule} we obtain that the spectral sequence  \begin{equation}
    \lim_{q\in \mathcal{Q}}{ }^i H^j(BG(q))\Rightarrow H^{i+j}(|\mathcal P|_{hK}).
\end{equation} is a spectral sequence of $C$-modules. Moreover, we show that it is a spectral sequence of finite $C$-modules. For every $q\in\mathcal{Q}$, we proved that $H^*BG(q)$ is a finitely generated $C$-module in Proposition \ref{prop:finiteCmodule}.

The reason is that the higher limits $\lim_\mathcal{Q}^i H^kBG(q)$ are computed as the homolo\-gy of the complex
\begin{align*}
\prod_{q_i} H^*BG(q_i) \leftarrow \prod_{q_i\to q_j} H^*BG(q_j) \leftarrow \prod_{q_i\to q_j \to q_k} H^*BG(q_k) \leftarrow \cdots
\end{align*}
Also, every element in this sequence is a finitely generated $C$-module because $\mathcal{Q}$ is a finite poset. Moreover, the differentials are morphisms of $C$-modules. We proved we have compatible morphisms,
\[
C \subseteq H^*(|\mathcal{P}|_{hG}) \xrightarrow{\alpha_{q}} H^*BG(q)
\]
\[
C \subseteq H^*(|\mathcal{P}|_{hG}) \xrightarrow{\alpha_{q'}} H^*BG(q')
\]
Then if we have $\varphi^*\colon H^*BG(q) \to H^*BG(q')$ needs to 
make the diagram commute, so they are morphisms of  $C$-modules.

Then, the complex differentials are linear combinations of maps of type $\varphi^*$. Thus, they are also morphisms of  $C$-modules. And that means that $H^*BG(q)$ is a finitely generated $C$-module.

Moreover, the $E_2$-term is of finite $C$-modules. Since $C$ is Noetherian, $E_\infty$ is a also a finite $C$-module. Finally $H^*(|\mathcal P|_{hK})$ is a finite $C$-module. But since $C$ is a finitely generated $\mathbb{F}_p$-algebra, then $H^*(|\mathcal P|_{hK})$ is also a finitely generated $\mathbb{F}_p$-algebra. 
\end{proof}

\begin{Obs}
    The work of Rector \cite{Rector} applies in this case, where he studied the Noetherian unstable $\mathcal{A}_p$ algebras $K$, showing how they are described up to F-isomorphism by a finite category generalizing Quillen's category of elementary abelian $p$-subgroups. Moreover the spectrum of homogeneous prime ideals admit an stratification parametrized by such a category, in a similar way as in Quillen's strong stratification theorem.

    In that case, the Krull dimension is the maximum of the ranks of elementary abelian $p$-subgroups $V$ such that there is morphism of unstable algebras $f\in \Hom_\mathcal{K}(H^*(|\mathcal P|_{hK}),H^*(V))$ which exhibits $H^*(V)$ as a finite $H^*(|\mathcal P|_{hK})$-module.
\end{Obs}

\subsection{Examples of non-finite posets $\mathcal{Q}$: finiteness properties}

One of the key properties we used, is the fact that the colimit over a finite poset of Noetherian $\textrm{End}(V)$-sets is also a Noetherian $\textrm{End}(V)$-set (Proposition \ref{p:colimNoetherian}). We consider the case in which the poset is not necessarily finite, but the colimit still is.

We focus on a particular example of a colimit of finite groups over the poset of natural numbers. Thus, let $X\colon\mathbb{N}\to \textrm{End}(V)\text{-sets}$ be a functor such that $X(i)$ is a Noetherian $\End(V)$-set for every $i\in \mathbb{N}$. We ask ourselves in which cases we can affirm $\colim_{i\in\mathbb{N}} X(i)$ 
finite. 
\begin{Ex}
    Let $\mathbb{Z}/p^{\infty} \cong \mathbb{Z}[1/p]/\mathbb{Z}$ denote the union of the cyclic $p$-groups $\mathbb{Z}/p^n$ under the obvious inclusions.
Let $X\colon\mathbb{N}\to \textrm{End}(V)\text{-sets}$ defined be $X(i)=\Rep(V,G(i))$ where the corresponding simple complex of groups is $G\colon \mathbb{N} \to \textrm{InjGroups}$ is $G(i)=\mathbb{Z}/p^i$. Then $X(i)=\Rep(V,\mathbb{Z}/p^i)=Hom(V,\mathbb{Z}/p^i)$. Note that in this case $\hat{G}=\mathbb{Z}/p^\infty$. In this case, one can see that $\colim_{\mathbb{N}}X$ is finite since any morphism from $V$ will have image in the elements of order $p$ in $G(i)$, and they are all identify with the image of $G(1)$. Then, in fact,  $\colim_{\mathbb{N}}X=\Rep(V,\mathbb{Z}/p^\infty)$.
\end{Ex}

\begin{Ex}
Another example is given by considering the simple complex of groups over $\mathbb{N}$ given by $G(i)=(\mathbb Z/p)^i$ with first coordinate inclusions. The group colimit $\hat{G}$ is an infinite product of $\mathbb{Z}/p$. Like in the previous example consider for each $i$ $X(i)=\Rep(V,(\mathbb{Z}/p)^i)=Hom(V,(\mathbb{Z}/p)^i)$, which is also finite. One can see that we have proper inclusions $X(i)\subsetneq X(i+1)$:  for every $n\in\mathbb{N}$ there is $g\in \Hom(V,G(n))$ which does not factor through any morphism in $\Hom(V,G(n-k))$ for any $k>0$.  Then $\colim_{\mathbb{N}} X$ is not finite. But note also that $\Rep(V,\colim_{\mathbb N} G)$ is not finite.

\end{Ex}

Both previous examples correspond to abelian simple complexes of groups where representations and morphisms agree.
 
\begin{Ex}\label{ex:quaternion}
    
Let $G\colon \mathbb{N} \to \textrm{InjGroups}$ defined as $G(i)=Q_{2^i},$ the generalized quaternion group. Recall that we have group presentations:
$$\langle x, y \mid x^{2^{i}} = y^{4} = 1,\; x^{2^{i-1}} = y^{2},\; y^{-1}xy = x^{-1} \rangle .$$ In particular all the elements in $Q_{2^i}$ have order $2^k$ with $k\leq i$. For each $Q_{2^i}$ the only order $2$ elements are $1$ and $-1$, which are central in this group. Any morphism from an elementary abelian $2$-group will have image in the center. 
As before we define $X(i)=\Rep(V,Q_{2^i}),$ with $i\in \mathbb{N}$ and $i\geq 3$.  
As before we define $X(i)=\Rep(V,Q_{2^i}),$ with $i\in \mathbb{N}$ and $i\geq 3$.  In this case $\hat{G}=Q_{2^\infty}$, the infinite quaternion group which has only one conjugacy class of elementary abelian $2$-groups.

\end{Ex}

Another example of a non abelian simple complex of groups analogous to Example \ref{ex:quaternion} is the one described by the family of dihedral groups. However, both examples correspond to a family of locally finite groups. We first need to introduce some definitions, and use the paper \cite{broto2007discrete} as a reference.

\begin{Def}
A \emph{discrete toral group} is a group $P$, with a normal subgroup $P_0 \lhd P$, such that $P_0$ is isomorphic to a finite product of copies of $\mathbb{Z}/p_i^{\infty}$, thus:
$$P_0\cong \prod_{i=1}^{n} \mathbb{Z}/p_i^{\infty},$$
where each \(p_i\) is a prime (not necessarily equal to the others).
And $P/P_0$ is a finite group.
\end{Def}

\begin{Def}
A group $G$ is \emph{locally finite} if every finitely generated subgroup of $G$ is finite, and it is a \emph{locally finite $p$-group} if every finitely generated subgroup of $G$ is a finite $p$-group.
\end{Def}

\begin{Def}
A group $G$ is \emph{artinian} if every non-empty set of subgroups of $G$, partially ordered by inclusion, has a minimal element. Equivalently, $G$ is artinian if its subgroups satisfy the descending chain condition.
\end{Def}

\begin{Obs}\label{obs:countable}
  Let $G$ be a locally finite and artinian group. Then the same strategy of proof as in the proof of \cite[Proposition 1.2]{broto2007discrete} applies, and $G$ fits into an extension 
    \begin{equation}
        1\to A \to G \to \pi \to 1
    \end{equation}
    where $\pi$ is a finite group and $A$ is a finite product of groups of the form $\mathbb{Z}/q^n$ with $q$ a prime number and $1\leq n\leq \infty$. In that case we see that locally finite artinian groups are also countable and can be expressed as $G=\cup_{n\in \mathbb{N}} G(n)$ with $G(n)$ finite groups. 
  
\end{Obs}
    
\begin{Prop}\cite[See Prop 1.2]{broto2007discrete}
A group is a discrete toral group if and only if it is artinian and
a locally finite group.
\end{Prop}

The important property of this family of groups that we will use is the following, whose proof is the same as in \cite[Lemma 1.4]{broto2007discrete}.

\begin{Prop}\label{toralconjclass}
    Let $G$ be a discrete toral group, then $G$ contains finitely many conjugacy classes of elementary abelian subgroups.
\end{Prop}

Let $\varphi\colon V\to \hat{G}$ where $\hat{G}$ is locally finite artinian as a colimit group of the simple complex of groups obtained from its structure, see Remark \ref{obs:countable}. Then $\varphi(V)\leq \hat{G}=\bigcup G(n)$ can also be described as a colimit via $\varphi(V)\cap G(n)$, and by finiteness, we see that $\varphi(V)\leq G(N)$ for some $N$. That is, any $\varphi\colon V\to \hat{G}$ factors through $G(N)$ for some $N$. If we consider 

\begin{equation*}
    \Phi\colon\colim_{n\in \mathbb{N}} \Rep(V,G(n))\to \Rep(V,\hat{G}),
\end{equation*}

we just showed that $\Phi$ is surjective. Now assume $[f] \in \Rep(V,G(n))$ and $[g] \in \Rep(V,G(m))$ such that $\Phi([f])=\Phi([g])$: that is, $f$ and $g$ are conjugate by an element $x\in \hat{G}=\cup G(n)$, that is $x\in G(k)$ for a big enough $k$. If we take $N$ bigger than $n$,$m$ and $k$, we see $f$ and $g$ are conjugate as morphisms into $G(N)$ and therefore they represent the same element in $\colim_{n\in \mathbb{N}} \Rep(V,G(n))$. We just proved that 

\begin{Lema}
    Under the previous assumptions, $\Phi$ is a bijection. Therefore, the $\textrm{End}(V)$-set $\colim_{n\in \mathbb{N}} \Rep(V,G(n))$ is finite.
\end{Lema}

\begin{Ex}
    Let $\hat{G}=D_{2^\infty}$ as a group colimit of the simple complex of groups $G\colon \mathbb{N} \to \textrm{InjGroups}$ defined by the dihedral groups $G(n)=D_{2^n}:=<r,s\mid r^{2^n}=s^2=1, srs=r^{-1}>$.
\end{Ex}

\begin{Obs}
    It is  known that the mod $p$ cohomology of locally finite artinian groups is Noetherian (see \cite[Proposition 12.1]{Dwyer-Wilkerson}). The proof uses a transfer argument and the Serre spectral sequence applied to the short exact sequence in Remark \ref{obs:countable}. In this section we have viewed this groups as a group colimit for a simple complex of groups indexed on $\mathbb{N}$. We are under the assumptions of Theorem \ref{main}, then $H^*(|\mathcal{P}|_{h\hat{G}})$ is a finitely generated $\mathbb{F}_p$-algebra. In this case the associated poset $\mathcal{P}$ is a tree; and we recover the finiteness result using the same methods. 
\end{Obs}

\subsection{Cohomology with p-adic coefficients}

Understanding mod $p$ cohomology and the torsion in integral cohomology allows to get results about the cohomology with $p$-adic coefficients. In \cite{Andersen_2012}, the authors formally established the relations between these qualitative properties, and we follow their strategy. The main result is the following.

\begin{Teo}\label{padicprincipal}
      Let $G$ be a simple complex groups over a finite poset, with $\varphi \colon \hat{G}\to K$ and an associated $K$-poset $\mathcal P$ defined as in Definition \ref{def:associatedGspace}. Then $H^*(|\mathcal{P}|_{hG};\hat{\mathbb{Z}}_p)$ is a Noetherian $\hat{\mathbb{Z}}_p$-algebra if and only if $H^*(|\mathcal{P}|_{hG};\mathbb{F}_p)$ is a Noetherian $\mathbb{F}_p$-algebra.  
\end{Teo}

To prove this, we begin with the following lemma.

\begin{Lema}\label{padicLema}
    Under the assumptions of Theorem \ref{padicprincipal}, 
    \begin{enumerate}
        \item\label{one} $H^n(|\mathcal{P}|_{hG};\hat{\mathbb{Z}}_p)$ is a finitely generated $\hat{\mathbb{Z}}_p$-module, for every $n$.
        \item\label{two} $H^*(|\mathcal{P}|_{hG};\hat{\mathbb{Z}}_p)$ has bounded torsion.
    \end{enumerate}
\end{Lema}

\begin{proof}
For the proof of the first statement (\ref{one}), recall that $|\mathcal{P}|_{hG} \simeq \hocolim_\mathcal{Q} BG(q)$, and then there is a spectral sequence $$\lim_\mathcal{Q}{ }^iH^j(BG(q); \hat{\mathbb{Z}}_p) \Rightarrow H^{i+j}(|\mathcal{P}|_{hG};\hat{\mathbb{Z}}_p).$$

    Assuming that $\mathcal{Q}$ is a finite poset, then we know that there exists $N\in \mathbb{N}$ such that $\lim_\mathcal{Q}{}^iH^j(BG(q); \hat{\mathbb{Z}}_p)=0$ for every $i>N$.

    Moreover, each $\lim_{p\in \mathcal{Q}}^i H^j(BG(q); \hat{\mathbb{Z}}_p)$ is a finitely generated  $\hat{\mathbb{Z}}_p$-module since $H^j(BG(q); \hat{\mathbb{Z}}_p)$ is so, with $\hat{\mathbb{Z}}_p$ being a Noetherian ring. Thus, in particular, $E_\infty^{i, j}$ is also a finitely generated $R$-module since the spectral sequence collapses at a finite stage.

    Then, in the case we are studying, by the convergence of the spectral sequence, we can affirm that for $H^n(|\mathcal{P}|_{hG};\hat{\mathbb{Z}}_p)$ there exists a filtration such that $$E_\infty ^{n,0}=F^0\subseteq F^1\subseteq \dots \subseteq F^{n-1} \subseteq F^n=H^n(|\mathcal{P}|_{hG};\hat{\mathbb{Z}}_p),$$ and for every $0<k\leq n$ we have $F^{k-1}\subseteq F^k \twoheadrightarrow E_\infty^{n-k, k}$.
    So, in particular, we have $$E_\infty ^{n,0}=F^0\subseteq F^1 \twoheadrightarrow E_\infty^{n-1, 1},$$ and as $E_\infty ^{n,0}$ and $E_\infty ^{n-1,1}$ are finitely generated $\hat{\mathbb{Z}}_p$-modules, we get that $F^1$ is also a finitely generated $\hat{\mathbb{Z}}_p$-module. Using this argument recursively, as the filtration is finite, and 
    $$F^n=H^n(|\mathcal{P}|_{hG};\hat{\mathbb{Z}}_p),$$
    we obtain that $H^n(|\mathcal{P}|_{hG};\hat{\mathbb{Z}}_p)$ is a $\hat{\mathbb{Z}}_p$-module finitely generated.

    In addition, to prove (\ref{two}), we use the fact that the integer cohomology of a finite group has bounded torsion:   $|G(q)|\cdot H^j(BG(q);\hat{\mathbb{Z}}_p)=0$ for $j>0$. Let $M:= lcm(|G(q)|,q\in \mathcal{Q})$. Then, $M\cdot E_2^{i,j}=0$, thus, $E_2^{i,j}$ has bounded torsion for $j>0$. When  $j=0$, $E_2^{i,0}=H^i(|\mathcal{Q}|;R)$ where $\mathcal{Q}$ is a finite poset, therefore the torsion is also bounded. Moreover, we know that we have a finite number of differentials, and $E_\infty^{i,j}$ has also bounded torsion, and then reconstructing $H^n(|\mathcal{P}|_{hG};\hat{\mathbb{Z}}_p)$ through the filtration as before, we get that it also has bounded torsion since there is a finite number of extensions to be solved, and this number is independent from $n$ because at most there is a number $N$ of $F^k\neq 0$, so $H^*(|\mathcal{P}|_{hG};\hat{\mathbb{Z}}_p)$ has bounded torsion.
\end{proof}

\begin{Obs}
    It is worth noting that one can obtain a bound on the torsion using the exponents of $H^n(BG(q);\hat{\mathbb{Z}}_p)$ and $H^n(|\mathcal{Q}|;\hat{\mathbb{Z}}_p)$ for all $n$ and the length of the poset $\mathcal{Q}$ since this last number measures the number of extensions needed to reconstruct the cohomology from the graded module associated to the filtration.
\end{Obs}

\begin{proof}[Proof of Theorem \ref{padicprincipal}]
By definition, $|\mathcal{P}|_{hG}$ is a connected space, and on Lemma \ref{padicLema}, we proved that $H^n(|\mathcal{P}|_{hG};\hat{\mathbb{Z}}_p)$ is a $\hat{\mathbb{Z}}_p$-module finitely generated. Then, $H^*(|\mathcal{P}|_{hG};\hat{\mathbb{Z}}_p)$ is Noetherian if, and only if, $H^*(|\mathcal{P}|_{hG};\mathbb{F}_p)$ is Noetherian and the torsion in $H^*(|\mathcal{P}|_{hG};\hat{\mathbb{Z}}_p)$ is bounded (see \cite[Theorem 2.4]{Andersen_2012}). But, by Lemma \ref{padicLema}, the torsion in $H^*(|\mathcal{P}|_{hG};\hat{\mathbb{Z}}_p)$ is bounded, so we get that $H^*(|\mathcal{P}|_{hG};\hat{\mathbb{Z}}_p)$ is Noetherian if, and only if, $H^*(|\mathcal{P}|_{hG};\mathbb{F}_p)$ is Noetherian.
\end{proof}

\bibliographystyle{alpha}
\bibliography{bibliography}

@article {BLO:homotopy_fusion_systems,
    AUTHOR = {Broto, Carles and Levi, Ran and Oliver, Bob},
     TITLE = {The homotopy theory of fusion systems},
   JOURNAL = {J. Amer. Math. Soc.},
  FJOURNAL = {Journal of the American Mathematical Society},
    VOLUME = {16},
      YEAR = {2003},
    NUMBER = {4},
     PAGES = {779--856},
      ISSN = {0894-0347,1088-6834},
   MRCLASS = {55R35 (20D15 55R40)},
  MRNUMBER = {1992826},
MRREVIEWER = {David\ J.\ Green},
       DOI = {10.1090/S0894-0347-03-00434-X},
       URL = {https://doi.org/10.1090/S0894-0347-03-00434-X},
}

@article {BK:Kac-Moody_groups,
    AUTHOR = {Broto, Carles and Kitchloo, Nitu},
     TITLE = {Classifying spaces of {K}ac-{M}oody groups},
   JOURNAL = {Math. Z.},
  FJOURNAL = {Mathematische Zeitschrift},
    VOLUME = {240},
      YEAR = {2002},
    NUMBER = {3},
     PAGES = {621--649},
      ISSN = {0025-5874,1432-1823},
   MRCLASS = {55R35 (22E65 55R40)},
  MRNUMBER = {1924024},
MRREVIEWER = {Lionel\ Schwartz},
       DOI = {10.1007/s002090100391},
       URL = {https://doi.org/10.1007/s002090100391},
}

@book {BH,
    AUTHOR = {Bridson, Martin R. and Haefliger, Andr\'{e}},
     TITLE = {Metric spaces of non-positive curvature},
    SERIES = {Grundlehren der mathematischen Wissenschaften [Fundamental
              Principles of Mathematical Sciences]},
    VOLUME = {319},
 PUBLISHER = {Springer-Verlag, Berlin},
      YEAR = {1999},
     PAGES = {xxii+643},
      ISBN = {3-540-64324-9},
   MRCLASS = {53C23 (20F65 53C70 57M07)},
  MRNUMBER = {1744486},
MRREVIEWER = {Athanase\ Papadopoulos},
       DOI = {10.1007/978-3-662-12494-9},
       URL = {https://doi.org/10.1007/978-3-662-12494-9},
}

@article {CS,
    AUTHOR = {Chach\'olski, Wojciech and Scherer, J\'er\^ome},
     TITLE = {Homotopy theory of diagrams},
   JOURNAL = {Mem. Amer. Math. Soc.},
  FJOURNAL = {Memoirs of the American Mathematical Society},
    VOLUME = {155},
      YEAR = {2002},
    NUMBER = {736},
     PAGES = {x+90},
      ISSN = {0065-9266,1947-6221},
   MRCLASS = {55P65 (18G55 55U30)},
  MRNUMBER = {1879153},
MRREVIEWER = {Timothy\ Porter},
       DOI = {10.1090/memo/0736},
       URL = {https://doi.org/10.1090/memo/0736},
}

@article {Dwyer-Wilkerson,
    AUTHOR = {Dwyer, W. G. and Wilkerson, C. W.},
     TITLE = {Homotopy fixed-point methods for {L}ie groups and finite loop
              spaces},
   JOURNAL = {Ann. of Math. (2)},
  FJOURNAL = {Annals of Mathematics. Second Series},
    VOLUME = {139},
      YEAR = {1994},
    NUMBER = {2},
     PAGES = {395--442},
      ISSN = {0003-486X,1939-8980},
   MRCLASS = {55R35 (55P35)},
  MRNUMBER = {1274096},
MRREVIEWER = {Haynes\ R.\ Miller},
       DOI = {10.2307/2946585},
       URL = {https://doi.org/10.2307/2946585},
}

@article {HLSmodulonilpotents,
    AUTHOR = {Henn, Hans-Werner and Lannes, Jean and Schwartz, Lionel},
     TITLE = {The categories of unstable modules and unstable algebras over
              the {S}teenrod algebra modulo nilpotent objects},
   JOURNAL = {Amer. J. Math.},
  FJOURNAL = {American Journal of Mathematics},
    VOLUME = {115},
      YEAR = {1993},
    NUMBER = {5},
     PAGES = {1053--1106},
      ISSN = {0002-9327,1080-6377},
   MRCLASS = {55S10 (55U99)},
  MRNUMBER = {1246184},
MRREVIEWER = {Donald\ M.\ Davis},
       DOI = {10.2307/2375065},
       URL = {https://doi.org/10.2307/2375065},
}

@article {Rector,
    AUTHOR = {Rector, D. L.},
     TITLE = {Noetherian cohomology rings and finite loop spaces with
              torsion},
   JOURNAL = {J. Pure Appl. Algebra},
  FJOURNAL = {Journal of Pure and Applied Algebra},
    VOLUME = {32},
      YEAR = {1984},
    NUMBER = {2},
     PAGES = {191--217},
      ISSN = {0022-4049,1873-1376},
   MRCLASS = {55R40 (55P45 55S10)},
  MRNUMBER = {741965},
MRREVIEWER = {Richard\ Kane},
       DOI = {10.1016/0022-4049(84)90051-3},
       URL = {https://doi.org/10.1016/0022-4049(84)90051-3},
}

@article {Thomason,
    AUTHOR = {Thomason, R. W.},
     TITLE = {Homotopy colimits in the category of small categories},
   JOURNAL = {Math. Proc. Cambridge Philos. Soc.},
  FJOURNAL = {Mathematical Proceedings of the Cambridge Philosophical
              Society},
    VOLUME = {85},
      YEAR = {1979},
    NUMBER = {1},
     PAGES = {91--109},
      ISSN = {0305-0041,1469-8064},
   MRCLASS = {18F25},
  MRNUMBER = {510404},
MRREVIEWER = {Daniel\ R.\ Grayson},
       DOI = {10.1017/S0305004100055535},
       URL = {https://doi.org/10.1017/S0305004100055535},
}

@article {WZ,
    AUTHOR = {Welker, Volkmar and Ziegler, G\"{u}nter M. and
              \v{Z}ivaljevi\'{c}, Rade T.},
     TITLE = {Homotopy colimits---comparison lemmas for combinatorial
              applications},
   JOURNAL = {J. Reine Angew. Math.},
  FJOURNAL = {Journal f\"{u}r die Reine und Angewandte Mathematik. [Crelle's
              Journal]},
    VOLUME = {509},
      YEAR = {1999},
     PAGES = {117--149},
      ISSN = {0075-4102,1435-5345},
   MRCLASS = {55P99 (05B30)},
  MRNUMBER = {1679169},
MRREVIEWER = {R.\ M.\ Vogt},
       DOI = {10.1515/crll.1999.035},
       URL = {https://doi.org/10.1515/crll.1999.035},
}

@incollection {lannes1987cohomologie,
    AUTHOR = {Lannes, J.},
     TITLE = {Sur la cohomologie modulo {$p$} des {$p$}-groupes ab\'eliens
              \'el\'ementaires},
 BOOKTITLE = {Homotopy theory ({D}urham, 1985)},
    SERIES = {London Math. Soc. Lecture Note Ser.},
    VOLUME = {117},
     PAGES = {97--116},
 PUBLISHER = {Cambridge Univ. Press, Cambridge},
      YEAR = {1987},
      ISBN = {0-521-33946-4},
   MRCLASS = {55S10 (20J05 55R35 55T15 55U99)},
  MRNUMBER = {932261},
MRREVIEWER = {J.\ F.\ Adams},
}

@article {Andersen_2012,
    AUTHOR = {Andersen, Kasper K. S. and Castellana, Nat\`alia and Franjou,
              Vincent and Jeanneret, Alain and Scherer, J\'er\^ome},
     TITLE = {Spaces with {N}oetherian cohomology},
   JOURNAL = {Proc. Edinb. Math. Soc. (2)},
  FJOURNAL = {Proceedings of the Edinburgh Mathematical Society. Series II},
    VOLUME = {56},
      YEAR = {2013},
    NUMBER = {1},
     PAGES = {13--25},
      ISSN = {0013-0915,1464-3839},
   MRCLASS = {55U20 (55N25 55R35)},
  MRNUMBER = {3021402},
MRREVIEWER = {J.\ M.\ Boardman},
       DOI = {10.1017/S0013091512000193},
       URL = {https://doi.org/10.1017/S0013091512000193},
}

@article {quillen,
    AUTHOR = {Quillen, Daniel},
     TITLE = {The spectrum of an equivariant cohomology ring. {I}, {II}},
   JOURNAL = {Ann. of Math. (2)},
  FJOURNAL = {Annals of Mathematics. Second Series},
    VOLUME = {94},
      YEAR = {1971},
     PAGES = {549--572; ibid. (2) 94 (1971), 573--602},
      ISSN = {0003-486X},
   MRCLASS = {57F10 (20J05 57D85)},
  MRNUMBER = {298694},
MRREVIEWER = {Larry\ Smith},
       DOI = {10.2307/1970770},
       URL = {https://doi.org/10.2307/1970770},
}

@book {Schwartz_1994,
    AUTHOR = {Schwartz, Lionel},
     TITLE = {Unstable modules over the {S}teenrod algebra and {S}ullivan's
              fixed point set conjecture},
    SERIES = {Chicago Lectures in Mathematics},
 PUBLISHER = {University of Chicago Press, Chicago, IL},
      YEAR = {1994},
     PAGES = {x+229},
      ISBN = {0-226-74202-4; 0-226-74203-2},
   MRCLASS = {55S10 (55S37 55T15)},
  MRNUMBER = {1282727},
MRREVIEWER = {Donald\ M.\ Davis},
}

@article {LANNES1989153,
    AUTHOR = {Lannes, Jean and Schwartz, Lionel},
     TITLE = {Sur la structure des {$A$}-modules instables injectifs},
   JOURNAL = {Topology},
  FJOURNAL = {Topology. An International Journal of Mathematics},
    VOLUME = {28},
      YEAR = {1989},
    NUMBER = {2},
     PAGES = {153--169},
      ISSN = {0040-9383},
   MRCLASS = {55S10},
  MRNUMBER = {1003580},
MRREVIEWER = {John\ C.\ Harris},
       DOI = {10.1016/0040-9383(89)90018-9},
       URL = {https://doi.org/10.1016/0040-9383(89)90018-9},
}

@article {Powell,
    AUTHOR = {Powell, Geoffrey},
     TITLE = {Finite presheaves and {$\mathcal{A}$}-finite generation of unstable
              algebras mod nilpotents},
   JOURNAL = {Ann. Inst. Fourier (Grenoble)},
  FJOURNAL = {Universit\'e{} de Grenoble. Annales de l'Institut Fourier},
    VOLUME = {69},
      YEAR = {2019},
    NUMBER = {5},
     PAGES = {2169--2204},
      ISSN = {0373-0956,1777-5310},
   MRCLASS = {55S10 (16T05 18F20)},
  MRNUMBER = {4018258},
MRREVIEWER = {J\'er\^ome\ Scherer},
       DOI = {10.5802/aif.3292},
       URL = {https://doi.org/10.5802/aif.3292},
}

@article {broto2007discrete,
    AUTHOR = {Broto, Carles and Levi, Ran and Oliver, Bob},
     TITLE = {Discrete models for the {$p$}-local homotopy theory of compact
              {L}ie groups and {$p$}-compact groups},
   JOURNAL = {Geom. Topol.},
  FJOURNAL = {Geometry \& Topology},
    VOLUME = {11},
      YEAR = {2007},
     PAGES = {315--427},
      ISSN = {1465-3060,1364-0380},
   MRCLASS = {55R35 (20D20 20J99 55S15)},
  MRNUMBER = {2302494},
MRREVIEWER = {Haynes\ R.\ Miller},
       DOI = {10.2140/gt.2007.11.315},
       URL = {https://doi.org/10.2140/gt.2007.11.315},
}

@book {Evens_1991,
    AUTHOR = {Evens, Leonard},
     TITLE = {The cohomology of groups},
    SERIES = {Oxford Mathematical Monographs},
      NOTE = {Oxford Science Publications},
 PUBLISHER = {The Clarendon Press, Oxford University Press, New York},
      YEAR = {1991},
     PAGES = {xii+159},
      ISBN = {0-19-853580-5},
   MRCLASS = {20J06 (20-02)},
  MRNUMBER = {1144017},
MRREVIEWER = {U.\ Stammbach},
}

@book {Matsumura89,
    AUTHOR = {Matsumura, Hideyuki},
     TITLE = {Commutative ring theory},
    SERIES = {Cambridge Studies in Advanced Mathematics},
    VOLUME = {8},
   EDITION = {Second},
      NOTE = {Translated from the Japanese by M. Reid},
 PUBLISHER = {Cambridge University Press, Cambridge},
      YEAR = {1989},
     PAGES = {xiv+320},
      ISBN = {0-521-36764-6},
   MRCLASS = {13-01},
  MRNUMBER = {1011461},
}

@book {May67,
    AUTHOR = {May, J. Peter},
     TITLE = {Simplicial objects in algebraic topology},
    SERIES = {Van Nostrand Mathematical Studies},
    VOLUME = {No. 11},
 PUBLISHER = {D. Van Nostrand Co., Inc., Princeton, N.J.-Toronto,
              Ont.-London},
      YEAR = {1967},
     PAGES = {vi+161},
   MRCLASS = {55.40},
  MRNUMBER = {222892},
MRREVIEWER = {A.\ K.\ Bousfield},
}

\end{document}